\documentclass[10pt]{article}

\usepackage{mathptmx,amsmath,amscd,amssymb,amsthm,xspace}
\usepackage[all,tips]{xy}
\usepackage[dvips]{graphicx}
\usepackage{graphics,epsfig,wrapfig,verbatim,syntonly}
\usepackage{hyperref,amssymb,color, url,fancyhdr}

\theoremstyle{plain}

\newtheorem{thm}{Theorem}[section]

\newtheorem{cor}[thm]{Corollary}

\newtheorem{prop}[thm]{Proposition}

\newtheorem{lemma}[thm]{Lemma}

\theoremstyle{definition}
\newtheorem*{defn}{Definition}

\newtheorem*{rem}{Remark}

\DeclareMathOperator{\GL}{GL}\DeclareMathOperator{\PSL}{PSL}
\DeclareMathOperator{\SL}{SL}

\DeclareMathOperator{\Isom}{Isom}

\DeclareMathOperator{\Tr}{Tr}
\DeclareMathOperator{\PGL}{PGL}\DeclareMathOperator{\Tri}{Tri}
\DeclareMathOperator{\Stab}{Stab}

\newcommand{\abs}[1]{\left\vert#1\right\vert}
\newcommand{\set}[1]{\left\{#1\right\}}
\newcommand{\norm}[1]{\left\vert \left\vert #1\right\vert\right\vert}

\newcommand{\pr}[1]{\left( #1 \right) }
\newcommand{\su}{\subset}

\newcommand{\bu}{\bigcup}

\newcommand{\lra}{\longrightarrow}

\newcommand{\B}[1]{\ensuremath{\mathbf{#1}}}

\newcommand{\Cal}[1]{\ensuremath{\mathcal{#1}}}
\newcommand{\Fr}[1]{\ensuremath{\mathfrak{#1}}}

\newcommand{\Hy}{\mathbf{H}}
\newcommand{\N}{\mathbf{N}}
\newcommand{\Q}{\mathbf{Q}}
\newcommand{\R}{\mathbf{R}}
\newcommand{\Z}{\mathbf{Z}}

\fancypagestyle{plain}

\begin{document}
\bibliographystyle{plain}


\title{\textbf{Primitive geodesic lengths and \\ (almost) arithmetic progressions}}
\author{J.-F. Lafont\thanks{The Ohio State University, Columbus, OH. E-mail: \tt{jlafont@math.ohio-state.edu}} \, and D. B. McReynolds\thanks{Purdue University, West Lafayette, IN. E-mail: \tt{dmcreyno@math.purdue.edu}}}
\maketitle

\begin{abstract}
\noindent In this article, we investigate when the set of primitive geodesic lengths on a Riemannian manifold 
have arbitrarily long arithmetic progressions. We prove that in the space of negatively curved metrics, a metric 
having such arithmetic progressions is quite rare. We introduce almost arithmetic progressions, a coarsification of arithmetic progressions, and prove that every negatively curved, closed Riemannian manifold has arbitrarily long almost arithmetic progressions in its primitive length spectrum. Concerning genuine arithmetic progressions, we prove that every non-compact, locally symmetric, arithmetic manifold has arbitrarily long arithmetic progressions in its primitive length spectrum. We end with a conjectural characterization of arithmeticity in terms of arithmetic progressions in the primitive length spectrum. We also suggest an approach to a well known spectral rigidity problem based on the scarcity of manifolds with arithmetic progressions.
\end{abstract}

\section{Introduction}

Given a Riemannian manifold $M$, the associated geodesic length spectrum is an invariant of central importance. When the manifold $M$ is closed and equipped with a negatively curved metric, there are several results that show primitive, closed geodesics on $M$ play the role of primes in $\Z$ (or prime ideals in $\Cal{O}_K)$. Prime geodesic theorems like Huber \cite{Huber}, Margulis \cite{Margulis}, and Sarnak \cite{Sarnak} on growth rates of closed geodesics of length at most $t$ are strong analogs of the prime number theorem (see, for instance, also \cite{Dietmar}, \cite{PS}, \cite{SY}, and \cite{Stopple}). Sunada's construction of length isospectral manifolds \cite{Sunada} was inspired by a similar construction of non-isomorphic number fields with identical Dedekind $\zeta$--functions (see \cite{Perlis}). The Cebotarev density theorem has also been extended in various directions to lifting behavior of closed geodesics on finite covers (see \cite{Sunada2}). There are a myriad of additional results, and this article continues to delve deeper into this important theme. Let us start by introducing some basic terminology:

\begin{defn}
Let $(M,g)$ be a Riemannian orbifold, and $[g]$ a conjugacy class inside the orbifold fundamental group 
$\pi_1(M)$. We let $L_{[g]} \subset [0,\infty)$ consist of the lengths of all closed orbifold geodesics in $M$ which represent the conjugacy class $[g]$.  This could be empty if $M$ is non-compact, and if $M$ is a compact manifold (rather than orbifold), then $L_{[g]}$ takes values in $\R^+:=(0,\infty)$. The \textbf{length spectrum} of $(M,g)$ is the multiset $\mathcal L(M,g)$ obtained by taking the union of all the sets $L_{[g]}$, where $[g]$ ranges over all conjugacy classes in $M$.

We say a conjugacy class $[g]$ is \textbf{primitive} if the element $g$ is not a proper power of some other element (in particular $g$ must have infinite order). The \textbf{primitive length spectrum} of $(M,g)$ is the multiset $\mathcal L_p(M,g)$ obtained by taking the union of all the sets $L_{[g]}$, where $[g]$ ranges over all primitive conjugacy classes in $M$.
\end{defn}

\subsection{Arithmetic progressions}

Partially inspired by the analogy with primes, we are interested in understanding, for a closed Riemannian manifold $(M,g)$, the structure of the primitive length spectrum $\Cal{L}_p (M,g)$. Specifically, we would like to analyze whether or not the multiset of positive real numbers $\Cal{L}_p (M,g)$ contains arbitrarily long arithmetic progressions. 

\begin{defn}
We say that a multiset $S$ contains a \textbf{$k$--term arithmetic progression} if it contains a sequence of numbers 
$x_1 < x_2 < \cdots < x_k$ with the property that, for some suitable $a, b$, we have $x_j= aj+b$. 
\end{defn}

We will say a (multi)-set $S$ \textbf{has arithmetic progressions} if it contains $k$--term arithmetic progressions for all 
$k\geq 3$. We will say that a (multi)-set of positive numbers \textbf{has no arithmetic progressions} if it contains no 
$3$--term arithmetic progressions (and hence, no $k$--term arithmetic progression with $k\geq 3$). 
Note that we do not allow for {\it constant} arithmetic progressions -- so that multiplicity of
entries in $S$ are not detected by, and do not influence, our arithmetic progressions.
Our first result 
indicates that generically, the primitive length spectrum of a negatively curved manifold has no arithmetic progression.

\begin{thm}\label{noAP-G-delta-dense}
Let $M$ be a closed, smooth manifold and let $\mathcal M(M)$ denote the space of all negatively curved 
Riemannian metrics on $M$, equipped with the Lipschitz topology. If $\mathcal X(M) \subseteq \mathcal M(M)$ 
is the set of negatively curved metrics $g$ whose primitive length spectrum $\Cal{L}_p (M,g)$ has no arithmetic 
progression, then $\mathcal X(M)$ is a dense $G_\delta$ set inside $\mathcal M(M)$.
\end{thm}

Recall that any two Riemannian metrics $g,h$ on the manifold $M$ are automatically bi-Lipschitz equivalent to 
each other. Let $1\leq \lambda_0$ denote the infimum of the set of real numbers $\lambda$ such that there
exists a $\lambda$--bi-Lipschitz map $f_\lambda\colon (M, g) \to (M,g')$ homotopic to the identity map. The \textbf{Lipschitz distance} 
between $g,g'$ is defined to be $\log (\lambda_0)$, and the \textbf{Lipschitz topology} on the space of metrics is the topology 
induced by this metric. The key to establishing Theorem \ref{noAP-G-delta-dense} lies in showing that any negatively curved 
metric can be slightly perturbed to have no arithmetic progression:

\begin{thm}\label{perturb-no-AP}
Let $(M,g)$ be a negatively curved closed Riemannian manifold. For any $\epsilon >0$, there exists a new Riemannian metric $(M, \bar g)$ with the property that:
\begin{itemize}
\item $(M, \bar g)$ is negatively curved (hence $\bar g \in \mathcal M(M)$).
\item For any $v\in TM$, we have the estimate $(1-\epsilon)\norm{v}_{g} \leq \norm{v}_{\bar g} \leq \norm{v}_{g}$.
\item The corresponding length spectrum $\Cal{L}_p (M, \bar g)$ has no arithmetic progression.
\end{itemize}
In particular, the metric $\bar g$ lies in the subset $\mathcal X(M)$
\end{thm}

The proof of Theorem \ref{perturb-no-AP} will be carried out in Section \ref{section:no-AP}. Let us deduce Theorem \ref{noAP-G-delta-dense} from Theorem \ref{perturb-no-AP}.

\begin{proof}[Proof of Theorem  \ref{noAP-G-delta-dense}]
To begin, note that the second condition in Theorem \ref{perturb-no-AP} ensures that the identity map is a
$(1-\epsilon)^{-1}$--bi-Lipschitz map from $(M,g)$ to $(M, \bar g)$. Hence, by choosing $\epsilon$ small enough, we can arrange for the Lipschitz distance between $g$, $\bar g$ to be as small as we want. In particular, we can immediately conclude that $\mathcal X(M)$ is dense inside $\mathcal M(M)$. Since $M$ is compact, the set $[S^1, M]$ of free homotopy classes of loops in $M$ is countable (it corresponds to conjugacy classes of elements in the finitely generated group $\pi_1(M)$). Let  $\Tri(M)$ denote the set of ordered triples of distinct elements in $[S^1,M]$, which is still a countable set. Fix a triple ${t}:= (\gamma_1, \gamma_2, \gamma_3)\in \Tri(M)$ of elements in $[S^1,M]$. For any $g\in \mathcal M(M)$, we can measure the length of the $g$--geodesic in the free homotopy class represented by each $\gamma_i$. This yields a continuous function $L_{{t}}\colon \mathcal M(M) \to \R^3$ when $\mathcal M(M)$ is equipped with the Lipschitz metric. Consider the subset $A \subset \R^3$ consisting of all points whose three coordinates form a $3$--term arithmetic progression. Note that $A$ is a closed subset in $\R^3$, as it is just the union of the three hyperplanes $x+y=2z$, $x+z=2y$, and $y+z=2x$. Since $\R^3 \setminus A$ is open, so is $L_{{t}}^{-1}\big(\R^3 \setminus A\big) \subset \mathcal M(M)$. However, we have by definition that $\mathcal X(M) = \bigcap_{{t}\in T(M)}L_{{t}}^{-1}\big(\R^3 \setminus A\big)$ establishing that $\mathcal X(M)$ is a $G_{\delta}$ set.
\end{proof}

Our proof of Theorem \ref{noAP-G-delta-dense} is actually quite general, and can be used to show that, for any continuous finitary relation on the reals, one can find a dense $G_\delta$ set of negatively curved metrics whose primitive length spectrum {\it avoids} the relation (see Remark \ref{perturb-no-relation}). As a special case, one obtains a well-known result of Abraham \cite{Abraham} that there is a dense $G_\delta$ set of negatively curved metrics whose primitive length spectrum is multiplicity free.

Now Theorem \ref{noAP-G-delta-dense} tells us that, for negatively curved metrics, the property of having arithmetic progressions in the primitive length spectrum is quite rare. There are two different ways to interpret this result:

\begin{enumerate}
\item[(1)] Arithmetic progressions are the wrong structures to look for in the primitive length spectrum. 

\item[(2)] Negatively curved metrics whose primitive length spectrum have arithmetic progressions should be very special. 
\end{enumerate}

The rest of our results attempt to explore these two viewpoints. 

\subsection{Almost arithmetic progressions}

Let us start with the first point of view (1). Since the property of having arbitrarily long arithmetic progressions is 
easily lost under small perturbations of the metric (e.g. our Theorem \ref{perturb-no-AP}), we next consider a coarsification of this notion. 

\begin{defn}
A finite sequence $x_1 < \cdots  <x_k$ is a $k$--term \textbf{$\epsilon$--almost arithmetic progression}
($k\geq 2$, $\epsilon >0$) provided we have $\abs{\frac{x_i - x_{i-1}}{x_j-x_{j-1}} -1 } < \epsilon$ for all $i, j \in \{2, \ldots, k\}$.
\end{defn}

\begin{defn}
A multiset of real numbers $S\subset \R$ is said to have \textbf{almost arithmetic progressions} if, for every $\epsilon >0$ and $k\in \N$, the set $S$ contains a $k$--term $\epsilon$--almost arithmetic progression.
\end{defn}

We provide a large class of examples of Riemannian manifolds $(M,g)$ whose primitive length spectra 
$\Cal{L}_p(M,g)$ have almost arithmetic progressions.

\begin{thm}\label{MainAAPGen}
If $(M,g)$ is a closed Riemannian manifold with strictly negative sectional curvature, then $\Cal{L}_p(M,g)$ has almost arithmetic progressions.
\end{thm}

We will give two different proofs of Theorem \ref{MainAAPGen} in Section \ref{section:AAP-generic}. The first proof is geometric/dynamical, and uses the fact that the geodesic flow on the unit tangent bundle, being Anosov, satisfies the specification property. The second proof actually shows a more general result. Specifically, any set of real numbers that is asymptotically ``dense enough'' will contain almost arithmetic progressions. Theorem \ref{MainAAPGen} is then obtained from an application of Margulis' \cite{Margulis} work on the growth rate of the primitive geodesics. The second approach is based on the spirit of Szemer\'edi's Theorem \cite{Sz} (or more broadly the spirit of the Erd\"{o}s--Turan conjecture) that large sets should have arithmetic progressions. 

\subsection{Arithmetic manifolds and progressions}

Now we move to viewpoint (2) -- a manifold whose primitive length spectrum has arithmetic progressions should be special. We show that several arithmetic manifolds have primitive length spectra that have arithmetic progressions. In the moduli space of constant $(-1)$--curvature metrics on a closed surface, the arithmetic structures make up a finite set. One reason to believe that such manifolds would be singled out by this condition is, vaguely, that one expects solutions to extremal problems on surfaces to be arithmetic. For example, the Hurwitz surfaces, which maximize the size of the isometry group as a function of the genus, are always arithmetic; it is a consequence of the Riemann--Hurwitz formula that such surfaces are covers of the $(2,3,7)$--orbifold and consequently are arithmetic.

Note that a $3$--term arithmetic progression $x<y<z$ is a solution to the equation $x+z = 2y$, and similarly, a 
$k$--term arithmetic progression can be described as a solution to a set of linear equations in $k$ variables. 
Given a ``generic'' discrete subset of $\R^+$, one would not expect to find any solutions to this linear equation within the set, and hence would expect no arithmetic progressions. Requiring the primitive length spectrum to have arithmetic progressions forces it to contain infinitely many solutions to a linear system that generically has none. Of course, constant $(-1)$--curvature is  already a rather special class of negatively curved metrics. Even within this special class of metrics, a 3--term progression in the length spectrum is still a non-trivial condition on the space of $(-1)$--curvature metrics.  Our first result shows that non-compact arithmetic manifolds have arithmetic progressions.

\begin{thm}\label{MainArithmetic}
If $X$ is an irreducible, non-compact, locally symmetric, arithmetic orbifold such that $\widetilde{X}$ is of non-compact type, then $\Cal{L}_p(X)$ has arithmetic progressions.
\end{thm}

In Section \ref{ArithSec}, we prove Theorem \ref{MainArithmetic}, as well as some stronger results. For example, the following result shows that non-compact, arithmetic, hyperbolic 2--manifolds have an especially rich supply of arithmetic progression.

\begin{thm}\label{LowDimExist}
If $(M,g)$ is a non-compact, arithmetic, hyperbolic $2$--manifold, then given any $\ell\in \Cal{L}_p(M,g)$ and $k\in \N$, we can find $k$--term arithmetic progression in $\Cal{L}_p(M)$ such that each term is an integer multiple of $\ell$.
\end{thm}

The same result also holds for non-compact, arithmetic, hyperbolic 3--orbifolds.

\begin{thm}\label{Arith3Man}
If $(M,g)$ is a non-compact, arithmetic, hyperbolic $3$--manifold, then given any $\ell\in \Cal{L}_p(M,g)$ and $k\in \N$, we can find $k$--term arithmetic progression in $\Cal{L}_p(M)$ such that each term is an integer multiple of $\ell$.
\end{thm}

Theorem \ref{LowDimExist} also holds for other commensurability classes of non-compact, locally symmetric, arithmetic orbifolds (see Corollary \ref{GenArith3Man}). The non-compactness condition helps avoid some difficulties that could be overcome. Recently Miller \cite{Miller} extended Theorem \ref{MainArithmetic} to compact manifolds, and proved that all arithmetic manifolds satisfy the stronger conclusions of Theorem \ref{LowDimExist}. These geometric properties suggest an approach to proving the primitive length spectrum determines a locally symmetric metric either locally or globally in the space of Riemannian metrics. This determination or rigidity result would also require an upgrade of Theorem \ref{noAP-G-delta-dense}. We also provide a conjectural characterization of arithmeticity and discuss a few existing conjectural characterizations in Section \ref{Final}.

\paragraph*{Acknowledgements}
We would like to thank Ted Chinburg, Dave Constantine, Matt Emerton, Tom Farrell, Andrey Gogolev, 
David Goldberg, Chris Judge, Vita Kala, Grant Lakeland, Marcin Mazur, Nicholas Miller, Abhishek Parab, Ben Schmidt, and Dave Witte-Morris for conversations on the material in this paper. We would also like to thank the referees for several useful comments.

The first author was partially supported by the NSF, under grants DMS-1207782, DMS-1510640. The second author was partially support by the NSF, under grants DMS-1105710, DMS-1408458. 

\section{Arithmetic progressions are non-generic}\label{section:no-AP}

In this section, we prove Theorem \ref{perturb-no-AP}. Starting with a negatively curved closed Riemannian manifold $(M,g)$, we want to construct a perturbation $\overline{g}$ of the metric so that the primitive length spectrum $\Cal{L}_p (M, \overline{g})$ contains no arithmetic progressions. The basic idea of the proof is to enumerate the geodesics in $(M,g)$ according to their length. One then goes through the geodesics in order, and each time we see a geodesic whose length forms the third term of an arithmetic progression, we perturb the metric along the geodesic to destroy the corresponding $3$--term arithmetic progression. The perturbations are chosen to have smaller and smaller support and amplitude, so that they converge to a limiting Riemannian metric. The limiting metric will then have no arithmetic progressions. We now proceed to make this heuristic precise.

\subsection{Perturbing to kill a single arithmetic progression}

Given a negatively curved Riemannian manifold $(M,g)$, we index the set of primitive geodesic loops $\{\gamma_1, \gamma_2, \ldots \}$ according to the lengths. We now establish the basic building block for our metric perturbations.

\begin{prop}\label{perturb-near-geodesic}
Let $(M,g)$ be a negatively curved closed Riemannian manifold, $\gamma_k$ a primitive geodesic in $(M, g)$ 
of length $\ell(\gamma_k)=L$, and $\epsilon >0$ a given constant. Then one can construct a negatively curved Riemannian metric $(M, \overline{g})$ satisfying the following properties. Given a loop $\eta$, we denote by $\overline{\eta}$ the unique $\overline{g}$--geodesic loop freely homotopic to $\eta$, and $\ell$ (or $\overline{\ell}$) denotes the $g$--length (or $\overline{g}$--length) of any curve in $M$. The properties
we require are:
\begin{enumerate}
\item[(a)] For any vector $v\in TM$, we have $(1-\epsilon) \norm{v}_{g} \leq \norm{v}_{\overline{g}} \leq \norm{v}_{g}$.
Moreover, all derivatives of the metric $\overline{g}$ are $\epsilon$--close to the corresponding derivatives
of the metric $g$.
\item[(b)] For an appropriate point $p$, the metric $\overline{g}$ coincides with $g$ on the complement of the $\epsilon$--ball centered at $p$.
\item[(c)] We have $L - \epsilon \leq \overline{\ell}(\overline{\gamma_k}) <L$.
\item[(d)] If $i\neq k$ with $\ell(\gamma_i) \leq L$ then $\overline{\ell}(\overline{\gamma_i}) =\ell(\gamma_i)$.
\item[(e)] If $\ell(\gamma_i) > L$, then $\overline{\ell}(\overline{\gamma_i}) >L$.
\end{enumerate}
\end{prop}

\begin{proof}
To lighten the notation, we will denote by $\gamma:=\gamma_k$ the geodesic whose length we want to slightly decrease. Let
$\mathcal S :=\{\gamma_i ~:~ i\neq k , \ell(\gamma_i)\leq L \}$ denote the finite collection of closed geodesics who are shorter than 
$\gamma$ (whose lengths should be left unchanged). Note that any $\eta \in \mathcal S$ is distinct from $\gamma$, hence 
$\gamma \cap \eta$ is a finite set of points. Now choose $p\in \gamma$ which does not lie on any of the $\eta \in \mathcal S$, and let 
$\delta$ be smaller than the distance from $p$ to all of the $\eta \in \mathcal S$, smaller than $\epsilon /2$, and smaller than the 
injectivity radius of $(M,g)$. We will modify the metric $g$ within the $g$--metric ball $B(p; \delta)$ centered at $p$ of radius $\delta$. 
This will immediately ensure that property (b) is satisfied. Since the $g$--geodesics $\eta \in \mathcal S$ lie in the complement 
of $B(p; \delta)$, they will also be $\overline{g}$--geodesics. This verifies property (d).

For (e), since the length spectrum of a closed negatively curved Riemannian manifold is discrete, there is a $\delta^\prime>0$ with 
the property that for any $\eta$ with $\ell(\eta) > L$, we actually have $(1-\delta ^\prime) \ell(\eta) > L$. By shrinking $\delta^\prime$ 
if need be, we can assume that $\delta^\prime < \epsilon$. We modify the metric on $B(p; \delta)$ so that, for any $v\in TB(p; \delta)$, 
we have 
\begin{equation}\label{conf}
(1-\delta^\prime){\norm{v}_{g}} \leq \norm{v}_{\overline{g}} \leq {\norm{v}_{g}}.
\end{equation}
Since $\delta^\prime < \epsilon$, the first statement in property (a) will follow. Moreover, if $\eta$ is any closed $g$--geodesic, and 
$\overline{\eta}$ is the $\overline{g}$--geodesic freely homotopic to $\eta$, then we have the inequalities:
\begin{align}
\overline{\ell}(\overline{\eta}) & = \int _{S^1} \norm{\overline{\eta} ' (t) }_{\overline{g}} dt \geq (1-\delta^\prime)\int _{S^1} \norm{\overline{\eta}'(t)}_{g} dt \\
& = (1-\delta^\prime) \ell(\overline{\eta})  \geq (1-\delta^\prime) \ell(\eta)
\end{align}
Inequality (2) follows by applying (\ref{conf}) point-wise, while inequality (3) comes from the fact that $\eta$ is the $g$--geodesic 
freely homotopic to the loop $\overline{\eta}$. Hence, by the choice of $\delta^\prime$, 
$\overline{\ell}(\overline{\eta}) \geq (1-\delta^\prime) \ell(\eta) >L$, confirming property (e).  Note that this exact same argument, applied
to $\gamma$, also establishes property (c).

To complete the proof, we are left with explaining how to modify the metric on $B(p; \delta)$ in order to ensure property (a),
and in particular equation (\ref{conf}). We start by choosing a very small  $\delta ^{\prime \prime} < \delta /2$,  
which is also smaller than the normal injectivity radius of $\gamma$. We will focus on an exponential normal 
$\delta ^{\prime \prime}$--neighborhood of the geodesic $\gamma$ near the point $p$ (we can reparametrize so that 
$\gamma(0)=p$). Choose an orthonormal basis $\{e_1, \ldots e_n\}$ at the point $\gamma(0)$, with $e_1 = \gamma ^\prime (0)$, 
and parallel transport along $\gamma$ to obtain an orthonormal family of vector fields $E_1, \ldots E_n$ along $\gamma$. 
The vector fields $E_2, \ldots , E_n$ provides us with a diffeomorphism between the normal bundle $N\gamma_k$ of 
$\gamma_k|_{(-\delta ^{\prime \prime}, \delta ^{\prime \prime})}$ and $(- \delta ^{\prime \prime}, \delta ^{\prime \prime}) \times \R^{n-1}$. 
Let $D \subset \R^{n-1}$ denote the open ball of radius $\delta ^{\prime \prime}$, and using the exponential map, we obtain a 
neighborhood $N$ of the point $p$ which is diffeomorphic to $(- \delta ^{\prime \prime}, \delta ^{\prime \prime}) \times D$. We 
use this identification to parametrize $N$ via pairs $(t, z) \in (- \delta ^{\prime \prime}, \delta ^{\prime \prime}) \times D$. First, observe 
that this neighborhood $N$ comes equipped with a natural foliation, given by the individual slices $\{t\} \times D$. This is a smooth 
foliation by smooth codimension one submanifolds, and assigning to each point $q\in N$ the unit normal vector (in the positive 
$t$--direction) to the leaf through $q$, we obtain a smooth vector field $V$ defined on $N$. We can (locally) integrate this vector 
field near any point $q=(t_0, z_0)\in N$ to obtain a well-defined function $\tau\colon N\rightarrow \R$, defined in a neighborhood 
of $q$ (with initial condition given by $\tau \equiv 0$ on the leaf through $q$). Observe that, along the geodesic $\gamma$, we have 
that $\tau(t, 0) = t$, but that in general, $\tau(t, z)$ might not equal $t$. In this (local) parametrization near any point $q\in N$, our 
$g$--metric takes the form 
\begin{equation}\label{old-metric}
g = d\tau^2 + h_{t},
\end{equation}
where $h_{t}$ is a Riemannian metric on the leaf $\{t\} \times D$. We now change this metric on $N$. Pick a monotone smooth 
function $f\colon [0, \delta^{\prime \prime}] \to [{1-\delta ^\prime},1]$, which is identically $1$ in a neighborhood of $\delta ^{\prime \prime}$, 
and is identically ${1-\delta ^\prime}$ in a neighborhood of $0$. Recall that we had the freedom of choosing $\delta^\prime$ as small 
as we want. By further shrinking $\delta^\prime$ if need be, we can also arrange for the smooth function $f$ to have all order derivatives 
very close to $0$. There is a continuous function $r\colon N\rightarrow [0, \delta ^{\prime \prime})$ given by sending a point to its distance 
from the geodesic $\gamma$. We define a new metric in the neighborhood $N$ which is given in local coordinates by:
\begin{equation}\label{new-metric}
\overline{g}= f(r) f(t) d\tau^2 + h_{t}
\end{equation}
where $r$ denotes the distance to the geodesic $\gamma$ (i.e. the distance to the origin in the $D$ parameter).

Let us briefly describe in words this new metric. We are shrinking our original metric $g$ in the directions given by the $\tau$ parameter. 
In a small neighborhood of the point $p$, the $\tau$ parameter vector (which coincides with $\gamma^\prime$ along $\gamma$) is shrunk 
by a factor of $1-\delta^\prime$. As you move away from $p$ in the $t$ and $r$ directions, the $\tau$ parameter vector is shrunk by a 
smaller and smaller amount ($f$ gets closer to $1$), until you are far enough, at which point the metric coincides with the $g$--metric. 
By the choice of $\delta^{\prime \prime}$, this neighborhood $N$ is entirely contained in $B(p; \delta)$, hence our new metric 
$\overline{g}$ coincides with the original one outside of $B(p, \delta)$. The fact that equation (\ref{conf}) holds is easy to see. At 
any point $x=(t, z) \in N$ we can decompose any given tangent vector $\vec v\in T_xM$ as $\vec v = v_\tau \frac{d}{d\tau} + \vec v_z$, 
with $v_\tau \in \R$ and $\vec v_z \in T_{t,z}\big(\{t\}\times D\big)$. The original $g$--length of $\vec v$ is given by 
$\norm{\vec v}_g^2 = v_\tau^2 + \norm{\vec v_z}_{h_{t}}^2$, while the new $\overline{g}$--length of $\vec v$ is given by 
$\norm{\vec v}_{\overline{g}}^2 = f(t)f(r) v_\tau^2 + \norm{\vec v_z}_{h_{t}}^2$. Now the fact that the function $f$ takes values in the 
interval $[1-\delta ^\prime , 1]$ yields equation (\ref{conf}) (which gives the first statement in property (a)).

We note that the curvature operator can be expressed as a continuous function of the Riemannian metric and its derivatives. The 
metrics $\overline{g}$ and $g$ only differ on $N$, where they are given by equations (\ref{old-metric}) and (\ref{new-metric}) respectively. 
However, the function $f$ was chosen to have all derivatives very close to $0$. It follows that the metrics $\overline{g}$ and $g$ are close, 
as are all their derivatives (giving the second statement in property (a)). Hence their curvature operators (as well as their sectional 
curvatures) will correspondingly be close. Since $g$ is negatively curved and $M$ is compact, by choosing the parameters small 
enough, we can ensure that $\overline{g}$ is negatively curved. This concludes the proof of Proposition \ref{perturb-near-geodesic}
\end{proof}

\begin{rem}
The reader might find it instructive to think through Proposition \ref{perturb-near-geodesic} in the special case where $(M, g)$ is
a closed hyperbolic manifold. Our perturbation is length non-increasing, and shortens at least one geodesic, so $\text{Vol}(M, \overline {g})<
\text{Vol}(M,g)$. Since the volume is a topological invariant for hyperbolic manifolds, we note that $\overline {g}$ is no longer hyperbolic -- 
the curvatures in the perturbed region {\it must} change, while the curvatures outside remain identically $-1$. By choosing our constants
small, we can nevertheless arrange for the curvatures and the volume of $(M, \overline{g}$ to be as close as we want to those of the original 
hyperbolic metric.
\end{rem}

\subsection{Perturbations with no arithmetic progressions}

Finally, we have the necessary ingredients to prove Theorem \ref{perturb-no-AP}.

\begin{proof}[Proof of Theorem \ref{perturb-no-AP}]
Given our negatively curved closed Riemannian manifold $(M,g)$, we will inductively construct a sequence of negatively curved 
Riemannian metrics $g_i$, starting with $g_0=g$. We will denote by $\gamma^{(i)}_{k}$ the $k^{th}$ shortest primitive geodesic 
in the $g_i$--metric. To alleviate notation, let us denote by $\mathcal L_i$ the primitive length spectrum of $(M,g_i)$, which we 
think of as a non-decreasing function $\mathcal L_i\colon \N \to \R^+$. In particular, $\mathcal L_i(k)= \ell _i\big(\gamma^{(i)}_{k}$\big), 
the length of $\gamma^{(i)}_k$ in the $g_i$--metric. We will be given an arbitrary sequence $\{\epsilon_n\}_{n\in \N}$ satisfying 
$\lim \epsilon_n =0$.
For each $n \in \N$, the sequence of metrics $g_i$ will then be chosen to satisfy the following properties:
\begin{enumerate}
\item For all $i\geq n$, the functions $\mathcal L_i$ coincide on $\{1, \ldots, n\}$.
\item Each subset $\mathcal L_n(\{1, \ldots ,n\}) \subset \R^+$ contains no $3$--term arithmetic progressions.
\item Each $g_{n+1}\equiv g_n$ on the complement of a closed set $B_n$, where each $B_n$ is a 
(contractible) metric ball in the $g$--metric of radius strictly smaller than $\epsilon _n$, and the sets $B_n$ are pairwise disjoint.
\item On the balls $B_n$, we have that for all vectors $v\in TB_n$, $(1-\epsilon_n) \norm{v}_{g_n} \leq \norm{v}_{g_{n+1}} \leq \norm{v}_{g_n}$. Moreover, for each $n\in \N$, all derivatives of the metric ${g}_{n+1}$ are close to the corresponding
derivatives of the metric $g_n$.
\item For each $i>n$, we have that $\gamma^{(n)}_i \setminus \bu_{j=1}^n B_j \neq \emptyset$.
\item The sectional curvatures of the metrics $g_n$ are uniformly bounded away from zero, and uniformly bounded from below.
\end{enumerate}

\underline{{\bf Assertion:}} There is a sequence of metrics $g_n$ ($n\in \N$) satisfying properties (1)--(6).

Let us for the time being assume the {\bf Assertion}, and explain how to deduce Theorem \ref{perturb-no-AP}. The {\bf Assertion} 
provides us with a sequence of negatively curved Riemannian metrics on the manifold $M$. By choosing a sequence 
$\{\epsilon_n\}_{n\in \N}$ which decays to zero fast enough, it is easy to verify (using (3) and (4)) that these metrics converge 
uniformly to a limiting Riemannian metric $g_\infty$ on $M$. Moreover, this metric is negatively curved (see (6)), and has the 
property that $\Cal{L}_p(M, g_\infty)$ has no arithmetic progression. To see that there are no arithmetic progressions, we just 
need the following:

\underline{{\bf Claim:}} Choose any homotopically non-trivial loop $\gamma$ on $M$. Then there exists a $k$ such that the representative
of $\gamma$ in the $g_n$-metric is the $k^{th}$ shortest geodesic (for all sufficiently large $n$).

Assuming the {\bf Claim}, we show that $\Cal{L}_p(M, g_\infty)$ has no arithmetic progression. Given any three free homotopy 
classes of loops, the claim implies that for sufficiently large $n$, the $g_n$ geodesics representing these classes are the 
$i^{th}, j^{th}$, and $k^{th}$ shortest geodesics, where $i, j, k$ are independent of $n$. Property (2) ensures that the three real 
numbers $\mathcal L_n(i), \mathcal L_n(j), \mathcal L_n(k)$ do not form a $3$--term arithmetic progression. Property (1) ensures 
this property for the metrics $g_m$, for all $m\geq n$, and hence for the limiting metric $g_\infty$. It follows that $\Cal{L}_p(M, g_\infty)$ 
has no arithmetic progression.

\begin{proof}[Proof of \textbf{Claim}]
We proceed by contradiction. Note that, in our sequence of metrics $g_n$, properties (3) and (4) ensure
that the length of the $g_n$-geodesic in the given homotopy class can only {\it decrease} as $n$ goes to infinity. Let $L$ denote
the length of the $g_0$-geodesic in the homotopy class. If the claim fails, then for each $k$, we can find a corresponding metric
in our sequence, in which there are at least $k$ geodesics of length shorter than the geodesic in our given homotopy class -- and
hence shorter than $L$. This implies that for the $g_\infty$--metric on $M$, we have infinitely many geometrically distinct primitive 
geodesics whose lengths are uniformly bounded above by $L$. On the other hand, property (5) implies that $g_\infty$ has strictly 
negative curvature, so $\Cal{L}_p(M, g_\infty)$ must be a discrete multiset in $\R$. This contradiction establishes the {\bf Claim}. 
\end{proof}

\begin{proof}[Proof of \textbf{Assertion}]
By induction, let us assume that $g_n$ is given, and let us construct $g_{n+1}$. In order to lighten the notation, we will suppress the
superscripts on the geodesics $\gamma_i^{(n)}$ -- all geodesics in the rest of this proof will be with respect to the $g_n$-metric.

We consider the set $\mathcal L_n( \{1, \ldots, n+1\}) \subset \R^+$, and check whether or not it contains any arithmetic progression. 
If it does not, we set $g_{n+1} \equiv g_n$, $B_{n+1} = \emptyset$, and we are done. If it does contain an arithmetic progression, then 
from the induction hypothesis we know that it is necessarily a $3$--term arithmetic progression with last term  given by $\mathcal L_n(n+1)$, 
the length of the $g_n$--geodesic $\gamma_{n+1}$. From property (5), the complement $\gamma_{n+1} \setminus \bu _{j=1}^n B_j$ 
is a non-empty set and can be viewed as a collection of open subgeodesics of $\gamma_{n+1}$. As each of the sets 
$\gamma_{n+1} \cap \gamma_j$ ($j\neq n$) is finite, we can choose a point $p$ on $\gamma_{n+1} \setminus \bu _{j=1}^n B_j$ 
which does not lie on any of the geodesics $\gamma_j$ for $j\leq n$. We choose a parameter 
$\epsilon^\prime < \epsilon _{n}$, small enough so that the $\epsilon^\prime$--ball centered at 
$p$ is disjoint from $\pr{\bu _{j=1}^n B_j} \cup \pr{\bu _{j=1}^n \gamma_j}$. Note that, in view of property (3), on the complement of 
$\bu _{j=1}^n B_j$, we have that $g_n\equiv g_{n-1} \equiv \cdots \equiv g_0$. In particular, for $\epsilon^\prime$ small, the metric 
ball centered at $p$ will be independent of the metric used. Shrinking $\epsilon^\prime$ further if need be, we can apply Proposition 
\ref{perturb-near-geodesic} (with a parameter $\epsilon< \epsilon ^\prime$ to be determined below), obtaining a metric $g_{n+1}$ 
which differs from $g_n$ solely in the $\epsilon^\prime$--ball centered at $p$. We define $B_{n+1}$ to be the $\epsilon^\prime$--ball 
centered at $p$, and now proceed to verify properties (1)--(6) for the resulting metric.

\underline{\bf Property (1):} We need to check that the resulting length function $\mathcal L_{n+1}$ satisfies 
$\mathcal L_{n+1}(i) = \mathcal L_n(i)$ when $i\leq n$. However, this equality follows from statement (d) in 
Proposition \ref{perturb-near-geodesic}. 

\underline{\bf Property (2):} In view of property (1), we have an equality of sets 
$\mathcal L_{n+1}(\{1, \ldots , n\}) = \mathcal L_{n}(\{1, \ldots ,n\})$. By the inductive hypothesis, we know that there is no $3$--term 
arithmetic progression in this subset. Since the set $\mathcal L_{n+1}(\{1, \ldots , n\})$ is finite, there are only finitely many real numbers 
which can occur as the third term in a $3$--term arithmetic progression whose first two terms lie in $\mathcal L_{n+1}(\{1, \ldots , n\})$; 
let $T$ denote this finite set of real numbers, and observe that by hypothesis, $L:=\mathcal L_{n}(n+1) \in T$. Since $T$ is finite, we can 
choose $\epsilon < \epsilon ^\prime$ small enough so that we also have $[L -\epsilon, L) \cap T =\emptyset$. Then it follows from 
statements (c) and (e) in our Proposition \ref{perturb-near-geodesic} that $L-\epsilon \leq\mathcal L_{n+1}(n+1) < L$ and hence 
$\mathcal L_{n+1}(n+1) \not\in T$. Since $\mathcal L_{n+1}(n+1)$ cannot be the third term of an arithmetic progression, we conclude 
that the set $\mathcal L_{n+1}(\{ 1, \ldots , n+1\})$ contains no 3--term arithmetic progressions, verifying property (2).

\underline{\bf Property (3):} This follows from our choice of $\epsilon^\prime < \epsilon _n$ and point $p$, and property (b) in 
Proposition \ref{perturb-near-geodesic}.

\underline{\bf Property (4):} This follows from the corresponding property (a) in Proposition \ref{perturb-near-geodesic} 
(recall that $\epsilon < \epsilon_n$).

\underline{\bf Property (5):} This follows readily from property (3), which implies that the individual
$B_j$ are the path connected components of the set $\bu _{j=1}^n B_j$. So if the closed geodesic $\gamma_i$ was entirely 
contained in $\bu _{j=1}^n B_j$, it would have to be contained entirely inside a single $B_j$. However, such a containment is 
impossible, as $\gamma_i$ is homotopically non-trivial in $M$, while each $B_j$ is a contractible subspace of $M$.

\underline{\bf Property (6):} This is a consequence of property (4), as the curvature operator varies
continuously with respect to changes in the metric and its derivatives. By choosing the sequence 
$\{\epsilon_n\}_{n\in \N}$ to decay to zero fast enough, we can ensure that the change in sectional curvatures 
between successive $g_n$--metrics is slow enough to be uniformly bounded above and below by a pair of 
negative constants.

This completes the inductive construction required to verify the {\bf Assertion}.
\end{proof} 

Now that we've proven the {\bf Assertion}, the proof of Theorem \ref{perturb-no-AP} is complete. \end{proof}

\begin{rem}\label{perturb-no-relation}
Let $\mathcal R$ be an $r$--ary relation ($r\geq 2$) on the reals $\R$, having the property that if 
$(x_1, x_2, \ldots ,x_r)$ in $\mathcal R$, then $x_1\leq x_2 \leq \cdots \leq x_r$. Assume the relation $\mathcal R$ also has the property 
that, given any $x_1\leq x_2 \leq \cdots \leq x_{r-1}$, the set $\set{z ~:~ (x_1, \ldots , x_{r-1}, z)\in \mathcal R}$ is {\it finite}. Then the 
reader can easily see that the proof given above for Theorem \ref{noAP-G-delta-dense} also shows that there is a dense set of negatively 
curved metrics $g$ with the property that the primitive length spectrum $\mathcal L_p(M,g)$ contains no $r$--tuple satisfying the relation 
$\mathcal R$. In the special case where there exists a continuous function $F\colon \R ^r \to \R$ with the property that $(x_1, \ldots , x_r)$ 
is in $\mathcal R$ if and only if $x_1\leq x_2 \leq \cdots \leq x_r$ satisfies $F(x_1, \ldots , x_r)=0$, one also has that this dense set of 
negatively curved metrics is a $G_\delta$ set. Our Theorem \ref{noAP-G-delta-dense} corresponds to the $3$--ary relation given by 
zeroes of the linear equation $F(x, y, z) = x - 2y + z$. For another example, consider the $2$--ary relation corresponding to the zeroes of 
the linear equation $F(x, y) = x - y$. In this setting, we recover a well-known result of Abraham \cite{Abraham} -- that there is a dense 
$G_\delta$ set of negatively curved metrics on $M$ which have no multiplicities in the primitive length spectrum.
\end{rem}

\section{Almost arithmetic progressions are generic}\label{section:AAP-generic}

In this section, we give two proofs that almost arithmetic progressions can always be found in the
primitive length spectrum of negatively curved Riemannian manifolds.

\subsection{Almost arithmetic progression - the dynamical argument}

The first approach relies on the dynamics of the geodesic flow. Recall that closed geodesics in $M$ correspond to periodic orbits of the geodesic flow $\phi$ defined on the unit tangent bundle $T^1M$. In the case where $M$ is a closed negatively curved Riemannian manifold, it is well known that the geodesic flow is Anosov (see for instance \cite[\S 17.6]{hasselblatt-katok}). Our result is then a direct consequence of the following:

\begin{prop}\label{APs-Anosov-flows}
Let $X$ be a closed manifold supporting an Anosov flow $\phi$. Then for any $\epsilon >0$ and natural number
$k\geq 3$, there exists a $k$--term $\epsilon$--almost arithmetic progression $\tau_1< \ldots < \tau_k$ and 
corresponding periodic points $z_1, \ldots , z_k$ in $X$ with the property that each $z_i$ has minimal 
period $\tau_i$. 
\end{prop}

Before establishing this result, we recall that the Anosov flow on $X$ has the \textbf{specification property} (see
\cite[Section 18.3]{hasselblatt-katok} for a thorough discussion of this notion). This means that, given any $\delta >0$, there exists a real number $d>0$ with the following property. Given the following specification data:
\begin{itemize}
\item any two intervals $[0, b_1]$ and $[b_1+ d, b_2]$ in $\R$ (here $b_1, b_2$ are arbitrary positive 
real numbers satisfying $b_1 + d < b_2$),
\item a map $P\colon [0, b_1] \cup [b_1 +d, b_2]  \to X$ such that
$\phi^{t_2-t_1}\big(P(t_1)\big)=P(t_2)$ holds whenever $t_1, t_2 \in [0, b_1]$ and whenever 
$t_1, t_2 \in [b_1+d, b_2]$ (so that $P$ restricted to each of the two intervals defines a pair of $\phi$--orbits),
\end{itemize}
one can find a periodic point $x$, of period $s$, having the property that for all $t\in [0, b_1] \cup [b_1+ d, b_2]$ we have
$d\big(\phi^t(x) , P(t)\big) < \delta$ (so the periodic orbit $\delta$--shadows the two given pairs of orbits). Moreover, the period $s$ satisfies $\big|s -(b_2+d)\big| < \delta$ (though $s$ might not be the minimal period of the point $x$). We now use this specification property to establish the proposition.

\begin{proof}
We start by choosing a pair of distinct periodic orbits $\mathcal O_1$, $\mathcal O_2$ for the flow $\phi$, with minimal periods $A, B$ respectively (existence of distinct periodic orbits is a consequence of the Anosov property). Since the closed orbits are distinct, there is a $\delta$ with the property that the $\delta$--neighborhoods of the two orbits are disjoint. Corresponding to this $\delta$, we let $d>0$ be the real number provided by the specification property. We fix a pair of points $p_i \in \mathcal O_i$, and now explain how to produce some new periodic points.

Given an $n\in \N$, we consider the two intervals $[0, A]$ and $[A+d, nB + A + d]$ in $\R$. 
We define a map 
\[ P\colon [0, A] \cup [A+d, nB + A + d] \lra X \quad \text{by} \quad P(t)= \begin{cases} \phi^t(p_1) & t\in [0, A] \\ \phi^{t - A - d}(p_2) & t\in  [A+d, nB + A + d].
\end{cases} \]
From the specification property, there exists $x_n\in X$ and $s_n$ with $\phi^{s_n}(x_n)=x_n$
and $\big|s_n - (nB + A + 2d) \big| < \delta$ such that $d\big(\phi^t(x_n) , P(t)\big) < \delta$ holds for all $t$ in $[0, A] \cup [A+d, nB + A + d]$. We claim that if $n > (A + 2d + \delta)/B$, then $s_n$ is the minimal period of the point $x_n$. Indeed, under this hypothesis, the subinterval $[A+d, nB + A + d]$ is at least half the length of the period $s_n$. So if $s_n$ were not minimal, one could find $t_1 \in [0, A]$ and $t_2 \in [A + d, nB + A + d]$ with the property that $y:=\phi^{t_1}(x) = \phi^{t_2}(x)$. However, the shadowing property implies that $d\big(y, P(t_i)\big) = d\big(\phi^{t_i}(x), P(t_i)\big) < \delta$, which tells us that $y$ lies in the $\delta$--neighborhood of both sets $\mathcal O_1 = P\big([0, A]\big)$ and $\mathcal O_2 = P\big([A+d, nB + A + d]\big)$. This containment plainly contradicts the choice of $\delta$. We conclude that $s_n$ is indeed the minimal period of the point $x_n$. Now that we have found a sequence $\{x_n\}$ of periodic points, with minimal periods $\{s_n\}$ (when $n$ is sufficiently large), it is easy to find a $k$--term $\epsilon$--almost arithmetic progression. First, pick the integer $N$ to satisfy the inequality $N > \text{max} \left\{ \frac{4\delta + 2\delta \epsilon}{B\epsilon} , \frac{A + 2d + \delta}{B} \right\}$.  Setting $z_i := x_{iN}$ and $\tau_i:= s_{iN}$, we claim that the real numbers $\tau_1, \ldots , \tau_k$ forms the desired almost arithmetic progression. Indeed, the condition $N> \frac{A + 2d + \delta}{B}$ ensures that $\tau_i$ is the minimal period of the corresponding $x_i$. From the specification property, each $\tau_i$ satisfies the inequality $\abs{\tau_i - (iNB + A + 2d)}< \delta$. An elementary calculation shows that the ratio of any successive difference satisfies
\[ 1 - \epsilon < 1 - \frac{4\delta}{NB + 2\delta} < \left|\frac{\tau_{i+1} - \tau_i}{\tau_{j+1} - \tau_j}\right| 
< 1+ \frac{4\delta}{NB - 2\delta} < 1+ \epsilon \]
where the outer inequalities follow from $N>\frac{4\delta + 2\delta \epsilon}{B\epsilon}$. \end{proof}

\begin{rem}
There exist examples of Anosov flows that are distinct from the geodesic flow on the unit tangent bundle of a negatively curved manifold. For example, Eberlein \cite{E} has constructed an example of a closed non-positively curved Riemannian manifolds whose geodesic flow is Anosov, and which contain ``large'' open sets where the sectional curvature is identically zero. There are also examples of Anosov flows that do {\it not} come from geodesic flows, e.g. the suspension of an Anosov diffeomorphism on an odd dimensional manifold. 
\end{rem}

\begin{rem}
The proof of Proposition \ref{APs-Anosov-flows} only used the fact that Anosov flows on a compact manifold satisfy the specification property. The argument in the proof also works in a slightly more general setting, for flows that satisfy the {\it weak specification property}. In the specification property, the constant $d$ is the {\it transition time} for the orbit to move from shadowing the first orbit segment to shadowing the second orbit segment. The crucial point is that $d$ depends on $\epsilon$, but not on the choice of the orbit segments to be shadowed. In the weak specification property, one lets the transition times depend on the choice of orbit segments, but constrain them to be bounded above by a constant $D$ (which depends on $\epsilon$). It is easy to adapt the proof of Proposition \ref{APs-Anosov-flows} to see that, if $X$ is a compact space with a flow satisfying the weak specification property, then for any $k$ and $\epsilon$, one can find $k$ periodic points whose orbit lengths form a $k$--term $\epsilon$--almost arithmetic progression. In \cite{CLT}, Constantine, Lafont, and Thompson show that the geodesic flow on a compact locally CAT(-1) space satisfies the weak specification property (it is unknown whether these spaces satisfy the specification property). It follows that the primitive length spectrum for these spaces also have arbitrarily long $\epsilon$--almost arithmetic progressions for all $\epsilon >0$.
\end{rem}

\subsection{Almost arithmetic progression - the density argument}

An alternate route for showing that the primitive length spectrum $\Cal{L}_p(M,g)$ of a negatively curved Riemannian manifold has arbitrarily long almost arithmetic progressions is to exploit Margulis' work on the growth rate of this sequence. For a multiset $S\subset \R^+$ which is \textbf{discrete}, in that any bounded interval contains only finitely many elements of $S$, we define the associated \textbf{counting function}  to be $S(n):= \abs{ \set{ x\in S~:~x \leq n}}$. 

\begin{prop}\label{AAPGen}
If $S(x)$ has the property that there is some $t>0$ such that $\lim_{x\to \infty} \frac{S(x-t)}{S(x)}$ exists and is not equal $1$, then $S$ has almost arithmetic progressions.
\end{prop}

\begin{proof}
Given an $\epsilon >0$, we want to find an $\epsilon$--almost arithmetic progression of some given length $N$. Let us decompose $\R^+ = \bigcup _{k\in \N} \big( (k-1)t, kt \big]$, and form a subset $A \subset \N$ via $A := \{ k ~:~ S \cap \big((k-1)t, kt\big] \neq \emptyset \}$. We now argue that the set $A\subset \N$ is the complement of a finite subset of $\N$. If not, we could find an infinite sequence $k_i \subset \N$ with $k_i \not \in A$. From the definition of $A$, we have that for each of these $k_i$, the set $S\cap \big((k_i-1)t, k_i t\big]$ is empty. In terms of the counting function, this gives $S\big((k_i-1)t\big)= S(k_it)$. Now we divide
by $S(k_it)$ and take the limit, giving $\lim _{i\to \infty} \frac{S(k_it-t)}{S(k_it)} = 1$. However, this contradicts the fact that the limit $\lim_{x\to \infty} \frac{S(x-t)}{S(x)}$ exists and is not equal to $1$. So $\N \setminus A$ is a finite set, as desired. Next we choose an $m$ sufficiently large so that all integers greater than or equal to $m$ lie in the set $A$, and moreover $1+ \frac {2}{\epsilon} < m$. Consider the sequence of natural numbers $\{ m, 2m, \ldots Nm\}$. Since each of these natural numbers
lies in the set $A$, we can choose numbers $x_j \in S\cap \big((jm-1)t, (jm)t\big]$, giving us a sequence of
numbers $x_1 < x_2 < \cdots <x_N$ in the set $S$. We claim that this sequence forms an $\epsilon$--almost arithmetic progression of length $N$. It suffices to estimate the ratio of the successive differences. Note that for any index $j$, we have the obvious estimate on the difference $(m-1)t< \abs{x_{j+1} - x_{j}}<(m + 1)t$. Looking at the ratio between any two such successive differences, we obtain $1-\epsilon< \frac {m-1}{m+1}<\frac{\abs{x_{i+1} - x_{i}}}{\abs{x_{j+1} - x_{j}}}< \frac {m+1}{m-1}<1+\epsilon$, where the two outer inequalities follow from the fact that $1+ \frac {2}{\epsilon} < m$. This completes the proof of the proposition.
\end{proof}

A celebrated result of Margulis \cite{Margulis} establishes that, for a closed negatively curved manifold, the counting function for the primitive length spectrum has asymptotic growth rate $S(x) \sim \frac {e^{hx}}{hx}$,  where $h>0$ is the topological entropy of the geodesic flow on the unit tangent bundle. It is clear that, for any $t>0$, we have $\lim_{x\to \infty} \frac{S(x-t)}{S(x)} = \lim _{x\to \infty} \frac {e^{h(x-t)}hx}{e^{hx}h(x-t)}=e^{-ht}$, which is clearly not equal to $1$ since both $h>0, t>0$. In particular, Margulis' work in tandem with Proposition \ref{AAPGen} yields a second proof of Theorem \ref{MainAAPGen}. 

\begin{rem}
Margulis' thesis actually establishes the asymptotics for the number of periodic orbits of Anosov flows.  Hence, appealing to Margulis, one can recover Proposition \ref{APs-Anosov-flows} as a special case of Proposition \ref{AAPGen}. We chose to still include our proof of Proposition \ref{APs-Anosov-flows} for three reasons. First, it is relatively elementary, using only the specification property for Anosov flows, rather than the sophisticated result in Margulis' thesis. Secondly, it is constructive, allowing us to concretely ``see'' the sequence of periodic orbits whose lengths form the desired almost arithmetic progression. Thirdly, Margulis' asymptotics are not known to follow directly from the specification (or weak specification) property, so the method of proof of Proposition \ref{APs-Anosov-flows} could cover examples not addressed by Proposition \ref{AAPGen}.
\end{rem}

\section{Arithmetic orbifolds}\label{ArithSec}

In this section, we study the property of having genuine arithmetic progressions in the primitive length spectrum. We first show that this property is invariant under covering maps. Next, we prove that certain arithmetic manifolds have arithmetic progressions in their primitive length spectrum.

\subsection{Commensurability invariance}

\begin{prop}\label{CommInv}
Given a finite orbifold cover $(\overline{M}, \overline{g})$ of an orbifold $(M, g)$ with covering map 
$p \colon \overline{M}\rightarrow M$, the following two statements are equivalent:
\begin{enumerate}
\item[(a)] The primitive length spectrum $\Cal{L}_p(M, g)$ has arithmetic progressions.
\item[(b)] The primitive length spectrum $\Cal{L}_p(\overline{M}, \overline{g})$ has arithmetic progressions.
\end{enumerate}
\end{prop}

\begin{proof}
We start by making a simple observation. For a closed curve $\gamma\colon S^1\rightarrow M$, we call a curve 
$\overline{\gamma}\colon S^1\rightarrow \overline{M}$ a lift of $\gamma$ if there is a standard covering map 
$q \colon S^1\rightarrow S^1$ (given by $z\mapsto z^n$) with the property that 
$\gamma\circ q \equiv p \circ \overline{\gamma}$. If $\gamma$ is a primitive geodesic in $M$, we observe that all of its 
lifts $\overline{\gamma}$ to $\overline{M}$ are also primitive geodesics. If $d$ is the degree of the cover 
$p\colon \overline{M} \rightarrow M$, then the lift $\overline{\gamma}$ will always have length that is an integral multiple 
of $\gamma$. Moreover, $1\leq \ell(\overline{\gamma})/ \ell(\gamma) \leq d$, for any geodesic $\gamma$ on $M$ and any lift $\overline{\gamma}$ of $\gamma$ to $\overline{M}$. For the direct implication that (a) implies (b), we assume that $\Cal{L}_p(M, g)$ contains arithmetic progressions. Fixing some $k\geq 3$, our goal is to find a $k$--term arithmetic progression in the set $\Cal{L}_p(\overline{M}, \overline{g})$. From Van der Waerden's theorem (see for instance \cite{van} or \cite{GR}), there is an integer $N:=N(d, k)$, so that if the set $\{1, \ldots , N\}$ is $d$--colored, it contains a $k$--term monochromatic arithmetic progression. Since $\Cal{L}_p(M, g)$ contains arithmetic progressions, we can find a collection of primitive closed geodesics $\gamma_1, \ldots , \gamma_N$ such that the corresponding real numbers $\ell(\gamma_1), \ldots , \ell(\gamma_N)$ form an $N$--term arithmetic progression. For each $\gamma_i$, choose a lift $\overline{\gamma_i}$ inside $\overline{M}$, and color the integer $i$ by the color 
$\ell(\overline{\gamma_i})/\ell(\gamma_i)$. Looking at the monochromatic indices that form an arithmetic 
progression, we see that the corresponding $\ell(\gamma_i)$ form a $k$--term arithmetic progression. Moreover, by construction, the corresponding lifts $\overline{\gamma_i}$ are primitive geodesics whose lengths 
$\ell(\overline{\gamma_i}) = m\cdot \ell(\gamma_i)$. Here $m$ is a fixed integer which we view as the color of 
the monochromatic sequence. This gives the desired $k$--term arithmetic progression in the set $\Cal{L}_p(\overline{M}, \overline{g})$.

For the converse implication, we assume (b), that $\Cal{L}_p(\overline{M},\overline{g})$ has arithmetic progressions. Given a primitive closed geodesic $\overline{\gamma}$ in $\overline{M}$, one can look at the image geodesic $p \circ \overline{\gamma}$ in $M$, and ask whether or not this geodesic is primitive. Since $\overline{\gamma}$ is primitive, the only way $p\circ \overline{\gamma}$ could fail to be primitive is if the map $p$ induced a non-trivial covering from $\overline{\gamma}$ to the image curve $p \circ \overline{\gamma}$. Of course, the degree $d_{\overline{\gamma}}$ of this covering is smaller than or equal to $d$, and the quotient curve will be a primitive geodesic $\gamma$ of length $\ell(\overline{\gamma})/d_{\overline{\gamma}}$. Now as before, to produce a $k$--term arithmetic progression in $\Cal{L}_p({M},{g})$, we let $N$ be the Van der Waerden number $N(d,k)$, and choose a sequence of primitive closed geodesics $\overline{\gamma}_1, \ldots ,\overline{\gamma}_N$ in $\overline{M}$ whose lengths form an arithmetic progression. For each of these, we consider the corresponding primitive closed geodesic $\gamma_i$ in $M$ of length $\ell(\overline{\gamma_i})/d_{\overline{\gamma_i}}$. We color the index $i$ according to the color $d_{\overline{\gamma}_i}$. Then from Van der Waerden's theorem, there is a monochromatic arithmetic subprogression $S\subset \{1, \ldots , N\}$. The corresponding family of primitive geodesics $\{\gamma_i\}_{i\in S}$ have lengths which form a $k$--term arithmetic progression inside $\Cal{L}_p({M},{g})$, as required.
\end{proof}

\begin{rem}
The argument in the proof of Proposition \ref{CommInv} applies almost verbatim in the setting of almost arithmetic progressions, and shows that the following two statements are also equivalent:
\begin{enumerate}
\item[(a)] The primitive length spectrum $\Cal{L}_p(M, g)$ has almost arithmetic progressions.
\item[(b)] The primitive length spectrum $\Cal{L}_p(\overline{M}, \overline{g})$ has almost arithmetic progressions.
\end{enumerate}
As we will not need this result, we leave the details to the interested reader.
\end{rem}

We record the following direct consequence of Proposition \ref{CommInv}.

\begin{cor}\label{Cor:CommInv}
If $M_1,M_2$ be commensurable, Riemannian orbifolds, then $M_1$ has arithmetic progression if and only if $M_2$ has arithmetic progressions.
\end{cor}

To prove Theorem \ref{MainArithmetic}, we require a slightly more technical result than Proposition \ref{CommInv}. We say that a primitive length $\ell \in \Cal{L}_p(M,g)$ \textbf{occurs in arithmetic progressions}, if for any $k$, there exists an integer $k$--term arithmetic progression $\set{a+bs}_{s=1}^{k} \su \N$ such that $\set{\ell(a+bs)}_{s=1}^{k} \su \Cal{L}_p(M,g)$. 

\begin{prop}\label{CommInv2}
For commensurable Riemannian orbifolds $(M,g), (M',g')$, the following are equivalent:
\begin{enumerate}
\item[(a)] Every primitive length in $\Cal{L}_p(M, g)$ occurs in arithmetic progressions.
\item[(b)] Every primitive length in $\Cal{L}_p(M', g')$ occurs in arithmetic progressions.
\end{enumerate}
\end{prop}

\begin{proof}
As both directions are logically equivalent, we will prove that (b) implies (a). We will assume that every primitive length in $\Cal{L}_p(M',g')$ occurs in arithmetic progressions. For each $\ell \in \Cal{L}_p(M)$ and for each $k \in \N$, we must provide $\set{\ell(a+bs)}_{s=1}^{k} \su \Cal{L}_p(M)$ with $a,b \in \N$. To that end, we will make two coloring arguments similar to that made in the proof of Proposition \ref{CommInv}. As $M,M'$ are commensurable, there is a common, finite Riemannian covering $M_0 \to M,M'$. Set $d_M,d_{M'}$ to be the degree of the covers $M_0 \to M,M'$, respectively and for any natural number $s$, let $\tau(s)$ be the number of positive divisors of $s$ (e.g. $\tau(p) = 2$ if $p$ is a prime). Set 
\[ D = \pr{\prod_{1 \leq d \leq d_M} d}\pr{\prod_{1 \leq d \leq d_{M'}} d}. \] 
By Van der Waerden's theorem, there is an integer $N_1$ with the property that any $\tau(d_M)$ coloring
of the set $\{1, \ldots, N_1\}$ contains a monochromatic $k$--term arithmetic progression, and there is an
integer $N_2$ such that any $\tau(d_{M'})$ coloring of the set $\{1, \ldots, N_2\}$ contains a monochromatic 
$N_1$--term arithmetric progression. 

Fix a closed lift to $M_0$ of the geodesic associated to $\ell$, which gives us a primitive geodesic in $M_0$ of length $j \ell$ for some divisor $j$ of $d_M$. This will descend to a (cover of a) primitive geodesic  on $M'$ of length $(j/i) \ell$ where $i$ is a divisor of $d_{M'}$. Since $\ell' = (j/i) \ell$ is the length of a primitive geodesic in $M'$, by assumption there is a constant $C:=C_{\ell',DN_2} \in \N$ such that 
\[ \set{C Dn \ell'}_{n=1}^{N_2} \su \set{C n \ell'}_{n=1}^{DN_2} \su \Cal{L}_p(M'). \] 

For each integer $ 1\leq n \leq N_2$, we take a primitive geodesic in $M'$ of length $CD n \ell'$, and look at a 
lift in $M_0$. The length of this lift will be $i_n\cdot CD n \ell'$, for some divisor $i_n$ of $d_{M'}$, and
we can color each integer $n$ in the set $\{1, \ldots, N_2\}$ by the corresponding $i_n$. This gives a coloring
of the set $\{1, \ldots, N_2\}$ by $\tau(d_{M'})$ colors, so from Van der Waerden's theorem, we can now extract a monochromatic $N_1$--term subsequence 
\[ \{a' + b' r\}_{r=1}^{N_1} \subset \{1, \ldots, N_2\}, \] 
corresponding to some fixed color $i_0$. Notice that this gives a sequence of $N_1$ primitive geodesics in $M_0$ with lengths 
\[ \{(CDi_0)(a'+b'r)\ell'\}_{r=1}^{N_1}. \] 
For each $r$, the corresponding primitive geodesic in $M_0$ projects back down to a (cover of a) primitive geodesic in $M$ of length 
\[ \frac{\big((CDi_0)(a'+b'r)\ell'\big)}{j_r} \] 
for some divisor $j_r$ of $d_M$. So we can color the set of indices $\{1, \ldots , N_1\}$ by the corresponding divisor $j_r$, giving us a coloring with $\tau(d_M)$ colors. By Van der Waerden's theorem, there exists a $k$--term monochromatic subsequence $\{a'' + b''s\}_{s=1}^k$ of indices, corresponding to some fixed color $j_0$. Looking at the corresponding primitive geodesics in $M$, we see that they have lengths given in terms of $s$ by the formula 
\[ \Big(\frac{CDi_0}{j_0}\Big)\big(a'+b'(a'' + b''s)\big)\ell'. \] 
Since $\ell' = (j/i) \ell$, we can substitute in and simplify the expression to obtain 
\[ \set{\Big(\frac{CDi_0j}{j_0i}\Big)\big((a'+ b'a'') + b'b''s)\big)\ell }_{s=1}^k \su \Cal{L}_p(M). \] 
Notice that all the constants appearing in the above expression are integers, and that moreover, the product $j_0i$ is a divisor of $D$. In particular, the following numbers 
\[ a = \Big(\frac{CDi_0j}{j_0i}\Big)(a'+ b'a''), \quad  b = \Big(\frac{CDi_0j}{j_0i}\Big)(b'b'') \] 
are integers, and we obtain a $k$--term arithmetic progression $\set{\ell(a+bs)}_{s=1}^{k} \su \Cal{L}_p(M)$.
\end{proof}

\subsection{The modular surface has arithmetic progressions}
 

\subsubsection{Preliminaries}

The closed geodesics $c_\gamma$ on $X$ are in bijection with the conjugacy classes $[\gamma]$ where $\gamma \in \PSL(2,\Z)$ is hyperbolic. The length of the geodesic $\ell(c_\gamma)$ and trace $\Tr(\gamma)$ are related via the formula (see \cite[p. 384]{MR}) 
\[ 2 \cosh\pr{\frac{\ell(c_\gamma)}{2}} =  | \Tr(\gamma) |. \] 
The geodesic $c_\gamma$ is primitive precisely when $\gamma$ is primitive in $\PSL(2,\Z)$. Up to the sign of the trace, the characteristic polynomial of $\gamma$ will be of the form $P_\gamma(t) = t^2 - | \Tr(\gamma)|t + 1$. As $\abs{\Tr(\gamma)}>2$ (see \cite[p. 51]{MR}), we see that $\lambda_\gamma$ is a real and $\Q(\lambda_\gamma) = K_\gamma/\Q$ is a real quadratic extension by the quadratic formula. Moreover, $\lambda_\gamma \in \Cal{O}_{K_\gamma}^1$ is a unit and $\lambda_\gamma^{-1}$ is the Galois conjugate of $\lambda_\gamma$. By Dirichlet's Unit Theorem (see \cite[p. 142]{Marcus}), the group of units $\Cal{O}_{K_\gamma}^1$ of $\Cal{O}_{K_\gamma}$ is isomorphic to $\set{\pm 1} \times \Z$, where $\Z$ is generated by a fundamental unit. We will say that $\gamma$ is \textbf{absolutely primitive} if $\lambda_\gamma$ is a fundamental unit in $\Cal{O}_{K_\gamma}^1$. 

\begin{lemma}\label{AllFieldsLemma}
For any real quadratic extension $K/\Q$, there exists an absolutely primitive, hyperbolic $\gamma \in \PSL(2,\Z)$ with $K_\gamma = K$.
\end{lemma}

\begin{proof}
Let $K/\Q$ be a real quadratic extension with $\Z[a_1,a_2] = \Cal{O}_K$. For $\alpha \in K$, left multiplication on $K$ by $\alpha$ is a $\Q$--linear map. Using the $\Q$--basis $\set{a_1,a_2}$ for $K$, we obtain an injective $\Q$--algebra homomorphism $K \to \mathrm{M}(2,\Q)$ and injective group homomorphisms $K^\times \to \GL(2,\Q)$, $\Cal{O}_K^\times \to \GL(2,\Z)$. The group $\Cal{O}_K^1$ maps into $\SL(2,\Z)$ and the image of a fundamental unit is absolutely primitive and hyperbolic.
\end{proof}

\begin{lemma}\label{PowerLemma}
If $\gamma,\eta \in \PSL(2,\Z)$ are hyperbolic with $K_\gamma = K_\eta = K$, then there exist $j_\gamma,j_\eta \in \Z$ such that $\Tr(\gamma^{j_\gamma}) = \Tr(\eta^{j_\eta})$.
\end{lemma}

\begin{proof}
Each of $\gamma,\eta$ is conjugate to a diagonal matrix of the form 
\[ \begin{pmatrix} \lambda_\gamma & 0 \\ 0 & \lambda_\gamma^{-1} \end{pmatrix}, \quad \begin{pmatrix} \lambda_\eta & 0 \\ 0 & \lambda_\eta^{-1} \end{pmatrix}. \] 
We know $\mu_K^{t_\gamma} = \lambda_\gamma$ and $\mu_K^{t_\eta} = \lambda_\eta$ for some $t_\gamma,t_\eta \in \Z$ where $\mu_K \in \mathcal{O}_K^1$ is a fundamental unit. Setting $L = \mathrm{LCM}(t_\gamma,t_\eta)$, we take $j_\gamma = \frac{L}{t_\gamma}$, $j_\eta = \frac{L}{t_\eta}$.
\end{proof}

As a consequence of Lemma \ref{AllFieldsLemma} and Lemma \ref{PowerLemma}, we have the following result.

\begin{cor}\label{PrimitiveLemma}
If $\gamma \in \PSL(2,\Z)$ is absolutely primitive, then $\gamma$ is primitive. Moreover, if $\gamma$ is primitive, then there exists an absolutely primitive $\eta \in \PSL(2,\Z)$ such that $\Tr(\gamma) = \Tr(\eta^j)$ for some $j \in \N$.
\end{cor}

\subsubsection{Producing long progressions}

Taking $\Gamma = \PSL(2,\Z)$, for each $\eta \in \PGL(2,\Q)$, $\Gamma_\eta = (\eta \Gamma \eta^{-1}) \cap \Gamma$ is a finite index subgroup of $\Gamma$ and $\eta \Gamma \eta^{-1}$ (see \cite[Ch.~10]{Rag}). We define 
\[ P\colon \Gamma \times \PGL(2,\Q) \longrightarrow \N \]
by 
\begin{equation}\label{Eq:PowerDef}
P(\gamma,\eta) = \min\set{j \in \N~:~ (\eta \gamma \eta^{-1})^j \in \Gamma}
\end{equation}
and note that $P(\gamma,\eta) \leq [\Gamma: \Gamma_\eta]$. For a fixed $\gamma \in \Gamma$, we define
\begin{equation}\label{Eq:PowerSet}
\Cal{P}(\gamma) = \set{P(\gamma,\eta)~:~ \eta \in \PGL(2,\Q)} \subseteq \N.
\end{equation}
We set 
\begin{equation}\label{Eq:ThetaDef}
\theta_{\gamma,\eta} = \eta \gamma^{P(\gamma,\eta)} \eta^{-1} \in \Gamma
\end{equation}
and note that $\ell(c_{\theta_{\gamma,\eta}}) = P(\gamma,\eta)\ell(c_\gamma)$. 

\begin{thm}\label{GenConstructionSL2}
If $\gamma \in \PSL(2,\Z)$ is primitive and hyperbolic with associated geodesic length $\ell = \ell(c_\gamma)$ and $k \in \N$, then  there exists an arithmetic progression $\set{C_{\gamma,k} \ell n}_{n=1}^k \su \Cal{L}_p(X)$ where $C_{\gamma,k} \in \Q$. Moreover, there exists $D_\gamma \in \N$ such that $C_{\gamma,k}D_\gamma \in \N$ for all $k$.
\end{thm}

We will see that the failure of $C_{\gamma,k}$ to be an integer is controlled by the failure of $\gamma$ to be absolutely primitive. Specifically, Theorem \ref{GenConstructionSL2} is a consequence of the following result in combination with Corollary \ref{PrimitiveLemma}.

\begin{thm}\label{PrimitiveConstructionSL2}
If $\gamma \in \PSL(2,\Z)$ is absolutely primitive and hyperbolic with associated geodesic length $\ell$ and $k \in \N$, then there exists an arithmetic progression $\set{C_{\gamma,k} \ell n}_{n=1}^k \su \Cal{L}_p(X)$ where $C_{\gamma,k} \in \N$.
\end{thm}

Before proving Theorem \ref{PrimitiveConstructionSL2}, we make a comment about finding arithmetic progressions in $\PSL(2,\Z)$. 

\begin{rem}
For each $a \in \N$ with $a>2$, we have $\gamma_a = \begin{pmatrix} a & -1 \\ -1 & 0 \end{pmatrix} \in \Gamma$ and every hyperbolic $\gamma \in \Gamma$ is conjugate in $\PGL(2,\Q)$ to $\gamma_a$. The eigenvalues for $\gamma_a$ are 
\[ \lambda = \frac{a \pm \sqrt{a^2 + 4}}{2}. \]
To obtain arithmetic progressions from the $\gamma_a$, one needs to prove that $\gamma_a$ is primitive in $\Gamma$. For small values of $a$, the eigenvalues of $\gamma_a$ are fundamental units and so $\gamma_a$ is primitive by Corollary \ref{PrimitiveLemma}. For $a=11$, the splitting field of the characteristic polynomial of $\gamma_{11}$ is $K = \Q(\sqrt{5})$. We have in this case that
\[ \lambda_{\gamma_{11}} = \frac{11 + 5\sqrt{5}}{2}, \quad \mu_K = \frac{1+\sqrt{5}}{2}. \]
In particular, $\lambda_{\gamma_{11}} = \mu_K^5$ and it could be the case that $\gamma_a = \eta^j$ for $j=5$ from some $\eta \in \Gamma$. Checking that $\gamma_{11}$ is primitive using the multinomial equations coming from the matrix entries from the equality $\gamma_{11} = \eta^5$ for a variable matrix
\[ \eta = \begin{pmatrix} x & y \\ w & z \end{pmatrix} \]
is not straightforward. For point of illustration, we can express the equality $\gamma_{11} =\eta^5$ instead as $\eta^{-2}\gamma_{11} = \eta^3$ and obtain five equations in $x,y,z,w$ (including $\det(\eta) = 1$):
\begin{align*}
x^3 + 2wxy + wxy &= 11wy + xy + 11z^2 + yz \\
wy^2 + x^2y + yz^2 + xyz &= -wy - z^2 \\
-wx^2 + w^2y + wz^2 + wxz &= -x^2 - 11wx - wy - wz \\
wxy + z^3 + 2wyz &= wx + wz \\
xz - yw & = 1.
\end{align*}
Varying $a$, the eigenvalues of $\gamma_a$ can be arbitrarily large powers of the fundamental unit and so one must verify that there are no solutions to the equation $\gamma_a = \eta^j$ for arbitrarily large $j$. One might instead use hyperbolic geometry. The geodesic axis stabilized by $\gamma_a$ must also be stabilized by $\eta$. Using the eigenvalues for $\gamma_a$, we can determine the two fixed points $x_-,x_+  \in \partial \mathbf{H}^2$ for $\gamma_a$ and then determine precisely which elements of $\Gamma$ also $x_-,x_+$. This also entails solving equations. We then must verify from these solutions that $\gamma_a$ generates the full stabilizer of this axis in $\Gamma$. Another geometric approach was suggested to us by Lakeland. The $\gamma_a$ have isometric circles (i.e. the set of points $x \in \B{H}^2$ such that $\abs{\gamma \, '(x)}=1$) with maximal radii. If $\gamma_a = \eta^j$ for $j>1$ and $\eta \in \Gamma$, then the radius for the isometric circle for $\eta$ would be strictly larger than $1$. To make this rigorous, one would need to prove that the $\gamma_a$ satisfy this extremal property with regard to the radii of their isometric circles. Assuming one has established that $\gamma_a$ is primitive, one can produce arithmetic progressions in the primitive length spectrum of the modular surface. Explicitly, for each $a>2$, we have the infinite arithmetic progression
\[ \set{\ell(c_{\gamma_a}), \ell(c_{\gamma_{\Tr(\gamma_a^2)}}),\ell(c_{\gamma_{\Tr(\gamma_a^3)}}),\ell(c_{\gamma_{\Tr(\gamma_a^4)}})\dots} \subset \mathcal{L}_p(X). \]
In combination with Corollary \ref{Cor:CommInv}, we conclude that all non-compact, arithmetic hyperbolic 2--orbifolds have arithmetic progressions. Moreover, for any primitive hyperbolic $\gamma \in \Gamma$, we have $\Tr(\gamma) = \pm \Tr(\gamma_a)$ for some $a >2$, and so every primitive length arises in arithmetic progressions also. 

Rather than attempt to check $\gamma_a$ is primitive, we give an alternative approach. Our method is elementary, using only linear algebra, modular arithmetic, and number theory. We replace the set of $\gamma_a$ with the set of absolutely primitive elements. Instead of establishing primitivity of the $\gamma_a$, we must find suitable conjugating elements for each absolutely primitive elements to produce $k$--term arithmetic sequences in $\mathcal{P}(\gamma)$. The conjugating elements we use are directly related to Hecke operators for the modular surface; they also do not depend on the specific absolutely primitive element. We note that both the set of $\gamma_a$ and the set of absolutely primitive elements satisfy a universal property. Every element of $\PSL(2,\Z)$ is conjugate in $\PGL(2,\Q)$ to a $\gamma_a$ while it is also conjugate in $\PGL(2,\Q)$ to a power of an absolutely primitive element; both $\gamma_a$ and the absolutely primitive element are also unique as $\Q$ has class number one. Our method also explicitly illustrates the underlying reason for why such progressions exist; the action of the commensurator $\PGL(2,\Q)$. That reason is the motivation for Conjecture A in \S 5 and the spectral isolation problem for locally symmetric metrics. Our method also has clear generalizations to other settings; Miller's subsequent work \cite{Miller} verifies that Theorem \ref{GenConstructionSL2} holds for \emph{all} arithmetic lattices (in the setting of Conjecture A), and utilizes this approach. 

\end{rem}


In order to produce arbitrarily long arithmetic progressions in $\Cal{L}_p(X)$, we proceed in two steps. 
\begin{enumerate}
\item[\textbf{Step 1.}]
For a hyperbolic $\gamma \in \Gamma$, we use the fixed collection $\set{\eta_m = \begin{pmatrix} 1 & 0 \\ 0 & m \end{pmatrix}} \su \PGL(2,\Q)$ to show that the set $\Cal{P}(\gamma)$ contains arbitrarily long arithmetic progressions; see \eqref{Eq:PowerSet} above. The hyperbolic elements are given by $\theta_{\gamma,\eta_m} = \eta_m \gamma^{P(\gamma,\eta_m)} \eta_m^{-1}$; see \eqref{Eq:PowerDef} and \eqref{Eq:ThetaDef} above.
\item[\textbf{Step 2.}]
We prove that when $\gamma \in \Gamma$ is hyperbolic and absolutely primitive, $\theta_{\gamma,\eta_m}$ primitive for all $m \in \N$.
\end{enumerate}

Step 1 breaks up into three sub-steps, starting first with the case when $m$ is prime and then proceeding to more intricate cases with respect to the prime factorization of $m$. The Chinese Remainder Theorem allows us to reduce to the case of prime powers. In the case of prime powers, the key fact in the production of arithmetic progressions is that the kernel of the homomorphism $\PSL(2,\Z/p^{j+1}\Z) \to \PSL(2,\Z/p^j\Z)$ induced by the ring homomorphism $\Z/p^{j+1}\Z \to \Z/p^j\Z$ is a $p$--group of exponent $p$. Step 2 is relatively straightforward and highlights the relevance of absolutely primitive, hyperbolic elements. 

\begin{proof}[Proof of Theorem \ref{PrimitiveConstructionSL2}]
For $\alpha \in \R$, we define $\eta_\alpha = \begin{pmatrix} 1 & 0 \\ 0 & \alpha \end{pmatrix}$ and note that $\eta_{\alpha^{-1}} = \eta_\alpha^{-1}$. Given $\gamma = \begin{pmatrix} a & b \\ c & d \end{pmatrix}$ and $m \in \N$, we see that
\[ \eta_m \gamma \eta_m^{-1} = \begin{pmatrix} 1 & 0 \\ 0 & m \end{pmatrix} \begin{pmatrix} a & b \\ c & d \end{pmatrix} \begin{pmatrix} 1 & 0 \\ 0 & m^{-1} \end{pmatrix} = \begin{pmatrix} a & m^{-1} b \\ mc & d \end{pmatrix} \]
and  $P(\gamma,\eta_m) = \min\set{j \in \N~:~ m \mid b_j}$ where $\gamma^j = \begin{pmatrix} a_j & b_j \\ c_j & d_j \end{pmatrix}$. Set
\[ \B{B}_L(\Z/m\Z) = \set{\begin{pmatrix} a & 0 \\ c & d \end{pmatrix}~:~a,c,d \in \Z/m\Z} < \PSL(2,\Z/m\Z). \]
We have the homomorphism $r_m\colon \Gamma \lra \PSL(2,\Z/m\Z)$ given by reducing the matrix coefficients modulo $m$ and $P(\gamma,\eta_m)$ is the smallest integer $j$ 
such that $r_m(\gamma^j) \in \B{B}_L(\Z/m\Z)$. Note that since $\gamma$ is hyperbolic, we have both 
$b,c \ne 0$ and for all $j\geq 1$, $b_j,c_j \ne 0$. Indeed, if this were not the case, then some power 
$\gamma^j$ would have either the form $\begin{pmatrix} a_j & 0 \\ c_j & d_j \end{pmatrix}$ or $\begin{pmatrix} a_j & b_j \\ 0 & d_j \end{pmatrix}$. Being an element of $\PSL(2,\Z)$ forces $a_j,d_j = \pm 1$ and thus $\gamma$ would be virtually unipotent, which is impossible since $\gamma$ is hyperbolic. 

\quad\quad \underline{\textbf{Step 1:}}~\textsl{Produce arithmetic progressions in $\mathcal{P}(\gamma)$ using $\eta_m$ for $m \in \N$.}

\smallskip\smallskip

\quad\quad\quad\quad \underline{\textbf{Step 1.1:}}~ \textsl{$m = p$ is prime.}
\smallskip\smallskip

As noted above, $P(\gamma,\eta_p)$ is the smallest power $j$ such that $r_p(\gamma^j) \in \B{B}_L(\B{F}_p)$. We have
\[ \abs{\PSL(2,\B{F}_p)} = \frac{p(p-1)(p+1)}{2}, \quad \abs{\B{B}_L(\B{F}_p)} = \frac{p(p-1)}{2} \] 
and so $P(\gamma,\eta_p) \leq p + 1$. 

\quad\quad\quad\quad \underline{\textbf{Step 1.2:}} ~\textsl{$m = p^k$ is a prime power.}

\smallskip\smallskip

As noted above, $P(\gamma,\eta_{p^k})$ is the smallest $j$ such that $r_{p^k}(\gamma^j) \in \B{B}_L(\Z/p^k\Z)$. We have the short exact sequence (see \cite[Cor 9.3]{Bass}, 
\cite[Ch.~9]{DSMS}, or \cite[Lem 16.4.5]{LS})
\begin{equation}\label{ExactSequence}
1 \lra V_{p} \lra \PSL(2,\Z/p^k\Z) \lra \PSL(2,\Z/p^{k-1}\Z) \lra 1,
\end{equation}
where $V_{p} \cong (\B{F}_{p}^3,+)$. We also have an exact sequence 
\[ 1 \lra W_{p} \lra \B{B}_L(\Z/p^k\Z) \lra \B{B}_L(\Z/p^{k-1}\Z) \lra 1, \]
where $W_{p} \cong (\B{F}_{p}^2,+)$. Since $(\B{F}_p^j,+)$ is an abelian group of exponent $p$ for any $j>0$, we have
\[ P(\gamma,\eta_{p^k}) = p^{s_k}P(\gamma,\eta_{p^{k-1}}) \] 
where $s_k = 0,1$. Thus, for 
\[ t_k = \sum_{n=2}^k s_n, \] 
we see that $P(\gamma,\eta_{p^k}) = p^{t_k} P(\gamma,\eta_{p})$, where $P(\gamma,\eta_{p}) \leq p + 1$. We require the following lemma.

\begin{lemma}\label{ResidualBorelLemma}
If $\tau \in \PSL(2,\Z)$ satisfies $r_{p^k}(\tau) \in \B{B}_L(\Z/p^k\Z)$ for all $k \in \N$, then $\tau \in \B{B}_L(\Z)$.
\end{lemma}

\begin{proof}
If $\tau = \begin{pmatrix} a & b \\ c & d \end{pmatrix} \in \PSL(2,\Z)$ is such that $r_{p^k}(\tau) \in \B{B}_L(\Z/p^k\Z)$ for all $k$, then $b=0$.
\end{proof}

As $\gamma$ is hyperbolic, $P(\gamma,\eta_{p^k}) = j_{k}$ is an unbounded sequence by Lemma \ref{ResidualBorelLemma}. Since $j_{k}$ is unbounded, there exists a subsequence $n_{t}$ such that $P(\gamma,\eta_{p^{n_{t}}}) = p^t P(\gamma,\eta_{p})$,  where $t$ ranges over $\N$. In particular, we have 
\[ \set{P(\gamma,\eta_{p}),pP(\gamma,\eta_{p}),p^2P(\gamma,\eta_{p}),p^3P(\gamma,\eta_{p}),\dots} \su \Cal{P}(\gamma). \] 

\quad\quad\quad\quad \underline{\textbf{Step 1.3:}}~\textsl{$m = p_1^{r_1}p_2^{r_2}\dots p_v^{r_v}$ is a product of primes to powers.}

\smallskip\smallskip

For distinct primes $p_1,\dots,p_v$ and $r_1,\dots,r_v \in \N$, set 
\[ m = \prod_{u=1}^v p_u^{r_u}. \] 
By the Chinese Remainder Theorem, there are isomorphisms
\begin{align*}
\PSL\pr{2,\Z/m\Z} &\cong \prod_{u=1}^v \PSL(2,\Z/p_u^{r_u}\Z) \\
\B{B}_L\pr{\Z/m\Z} &\cong \prod_{u=1}^n \B{B}_L(\Z/p_u^{r_u}\Z). 
\end{align*}
Thus, 
\[ P(\gamma,\eta_{m}) = \textrm{LCM}\set{P(\gamma,\eta_{p_1^{r_1}}),\dots,P(\gamma,\eta_{p_v^{r_v}})}. \] 
Since for each prime $p_i$, the sequence $P(\gamma,\eta_{p_i^k})$ is of the form $p_i^{t_k}P(\gamma,\eta_{p_i})$, we see that
\[ P(\gamma,\eta_{m}) = \pr{\prod_{u=1}^v p_u^{t_{r_u}}} \textrm{LCM}\set{P(\gamma,\eta_{p_1}),\dots,P(\gamma,\eta_{p_v})}. \]
For
\begin{equation}\label{ConstantEqu}
C_{\gamma, p_1,\dots,p_v} = \textrm{LCM}\set{P(\gamma,\eta_{p_1}),\dots,P(\gamma,\eta_{p_v})},
\end{equation}
we see that $\set{C_{\gamma, p_1,\dots,p_v}p_1^{w_1}\dots p_u^{w_u}}_{w_i\geq 0} \su \Cal{P}(\gamma)$, where $w_1,\dots,w_u$ range independently over all possible non-negative integers. From this fact, it is now a simple matter to produce arithmetic progressions in $\Cal{P}(\gamma)$.  Let $k \in \N$ and let $p_1,\dots,p_{u_k}$ to be all the prime divisors of the numbers $\set{1,\dots,k}$. Using these primes and setting $C_k:=C_{\gamma, p_1,\dots,p_{u_k}}$, the discussion in the previous paragraph yields
\[ \set{C_k, 2C_k, \ldots , kC_k} \su \set{C_k \cdot p_1^{w_1}\dots p_{u_k}^{w_{u_k}}}_{w_i\geq 0} 
\su \Cal{P}(\gamma). \]

Now, for each $1\leq r \leq k$, we have associated to the number $C_k r \in  \Cal{P}(\gamma)$ an element 
\[ \theta_{\gamma,\eta_r} = \eta_r \gamma^{C_k r} \eta_r^{-1} \in \PSL(2,\Z). \] 
The associated geodesic for $\theta_{\gamma,\eta_r}$ has length $\ell(c_{\theta_{\gamma,\eta_r}}) = C_k r \ell(c_\gamma)$. In particular, as $r$ ranges over $1\leq r \leq k$, we have a $k$--term arithmetic progression involving an integral multiple of the length of $\gamma$, where each of these lengths arises as the length of some closed geodesic. This completes the first step. 

\quad \quad \underline{\textbf{Step 2:}} \textsl{Prove $\theta_{\gamma,\eta}$ is primitive when $\gamma$ is absolutely primitive and $\eta \in \PGL(2,\Q)$.}

\smallskip\smallskip


To this end, let $\eta \in \PGL(2,\Q)$ and let $j = P(\gamma,\eta)$ with $\theta_{\gamma,\eta} = \eta \gamma^j \eta^{-1}$. By way of contradiction, assume there exists $\mu \in \PSL(2,\Z)$ with $\mu^{j'} = \theta_{\gamma,\eta}$. Diagonalizing via some $D \in \PGL(2,\R)$, we see that 
\[ D \mu^{j'} D^{-1} = D \theta_{\gamma,\eta} D^{-1} =  D \eta \gamma^j \eta^{-1} D^{-1} \] 
and
\[ \begin{pmatrix} \lambda_{\theta_{\gamma,\eta}} & 0 \\ 0 & \lambda_{\theta_{\gamma,\eta}}^{-1} \end{pmatrix} = \begin{pmatrix} \lambda_{\mu}^{j'} & 0 \\ 0 & \lambda_\mu^{-j'} \end{pmatrix} = \begin{pmatrix} \lambda_\gamma^j & 0 \\ 0 & \lambda_\gamma^{-j} \end{pmatrix}. \]
Since $\gamma$ is absolutely primitive, $\lambda_\mu = \lambda_\gamma^L$ for some $L \in \N$ and so
\[ D \mu D^{-1} = \begin{pmatrix} \lambda_\mu & 0 \\ 0 & \lambda_\mu^{-1} \end{pmatrix} = \begin{pmatrix} \lambda_\gamma^L & 0 \\ 0 & \lambda_\gamma^{-L} \end{pmatrix} = D \eta \gamma^L \eta^{-1} D^{-1}. \] 
Consequently, we have $\eta \gamma^L \eta^{-1} = \mu \in \PSL(2, \Z)$. As $j$ is the smallest power of $\gamma$ whose $\eta$--conjugate lands in $\PSL(2, \Z)$, we conclude that $L\geq j$. On the other hand, the fact that $\mu^{j'} = \theta_{\gamma,\eta}$ immediately tells us that $j'L = j$, which gives us $L\leq j$ since $j,L>0$. Hence $L=j$ and $j'=1$, and so $\theta_{\gamma,\eta}$ is primitive. 
\end{proof}

Since every non-compact, arithmetic, hyperbolic $2$--orbifold is commensurable with the modular surface (see \cite[Thm 8.2.7]{MR}), our work above in tandem with Corollary \ref{Cor:CommInv} yields:

\begin{cor}\label{Arith2Man}
If $M$ is a non-compact, arithmetic, hyperbolic $2$--orbifold, then $\Cal{L}_p(M)$ contains arithmetic progressions.
\end{cor}

\begin{rem}\label{estimate-constant}
The constant $C_{\gamma,k}$ is given by (\ref{ConstantEqu}), where the primes $p_i$ are all the possible prime divisors of $\set{1,\dots,k}$. Since $P(\gamma,\eta_{p_i}) \leq p_i+1$, we see that 
\[ C_{\gamma,k}  = \textrm{LCM}\set{P(\gamma,p)~:~p\text { is prime},~p \leq k} \leq \prod_{\substack{p \leq k, \\ p \text{ prime}}} (p+1). \]
\end{rem}

\subsection{Proof of Theorem \ref{LowDimExist}}\label{gen-and-imp}

Theorem \ref{LowDimExist} follows from Proposition \ref{CommInv2} and the following result. 

\begin{cor}\label{AllPrimArise}
Every primitive length for the modular surface occurs in arithmetic progressions
\end{cor}

\begin{proof}
Let $\ell' = \ell/D_\ell$ be the length of the associated absolutely primitive geodesic for the primitive length $\ell$. Set $S = \set{D_\ell,2D_\ell,\dots,kD_\ell}$ and let $\Cal{P}_S$ be the set of distinct prime factors for the elements of $S$. Using our construction above, we can find a constant $C_{\ell',S} \in \N$ such that $\set{C_{\ell',S}D_\ell n\ell'}_{n=1}^k \su \Cal{L}_p(X)$. For that, note that we can simply replace $S$ with the larger set $\set{1,\dots,kD_\ell}$  to produce the desired progression using the length $\ell'$ as in the proof of Theorem \ref{PrimitiveConstructionSL2}. Hence, we see that $C_{\ell',S}D_\ell n \ell' = C_{\ell',S} n \ell$ and so $\set{C_{\ell',S}n \ell} \su \Cal{L}_p(X)$.
\end{proof}

\subsection{Proof of Theorem \ref{MainArithmetic}}

We begin with the following straightforward lemma.

\begin{lemma}\label{tot-geod-gives-AP-general}
If $M, N$ are a pair of non-positively curved orbifolds and $N \hookrightarrow M$ is a locally isometric orbifold embedding, then we have an induced inclusion $\mathcal L_p(N) \hookrightarrow \mathcal L_p(M)$.
\end{lemma}

The following is a direct consequence of Lemma \ref{tot-geod-gives-AP-general} and Corollary \ref{Arith2Man}.

\begin{cor}\label{tot-geod-gives-AP}
Let $M$ be a non-positively curved manifold. If $M$ contains an embedded, totally geodesic submanifold commensurable with the modular surface, then $\mathcal L_p(M)$ has arithmetic progressions.
\end{cor}

We also need the following consequence of the Jacobson--Morosov Lemma:

\begin{lemma}\label{JM-Lemma}
If $M$ is an irreducible, non-compact, locally symmetric, arithmetic orbifold, then $M$ contains a totally geodesic suborbifold that is commensurable with the modular surface.
\end{lemma}

\begin{proof}
The hypotheses on $M$ imply that the orbifold fundamental group $\pi_1(M) = \Lambda$ is a lattice in a semisimple Lie group $G$  and $\Lambda$ is commensurable with $\B{G}(\Z)$, where $\B{G}$ is a $\Q$--defined semisimple Lie group isogenous to $G$ (see also \cite[5.27]{Witte}). As $M$ is non-compact, $\Lambda$ contains a non-trivial unipotent element by Godement's compactness criterion \cite{Godement} (see also \cite[5.26]{Witte}). By Jacobson--Morosov Lemma, $\B{G}$ has a $\Q$--defined subgroup $\B{G}_0$ that contains this non-trivial unipotent element and is isogenous to $\SL_2$ (see \cite[Lemma 4]{Jacobson}). The group $\B{G}_0(\Z) = \B{G}_0 \cap \B{G}(\Z)$ is an arithmetic lattice in $\B{G}_0(\R)$ by Borel--Harish-Chandra \cite{BHC} and is non-cocompact by Godement's compactness criterion. The subgroup $\B{G}_0 \cap \Lambda < \Lambda$ gives rise to a totally geodesic suborbifold that is commensurable with the modular surface. 
\end{proof}

\begin{rem}
In general, the group $\B{G}_0$ is not the stabilizer under the action of $\B{G}$ of the totally geodesic hyperbolic plane associated to $\B{G}_0$ in the symmetric space associated to $\B{G}$. The full stabilizer $\Stab_{\B{G}_0}$ in $\B{G}$ can also have a compact factor. In particular, the group $\Stab_{\B{G}_0} \cap \Lambda$ contains $\B{G}_0 \cap \Lambda$ as a finite index subgroup. 
\end{rem}

\begin{proof}[Proof of Theorem \ref{MainArithmetic}]
In order to prove Theorem \ref{MainArithmetic} with the above results, we require one additional step as the above submanifold is not necessarily embedded. The subgroup $\B{G}_0 \cap \Lambda$ gives rise to an immersed, totally geodesic suborbifold of the locally symmetric orbifold associated to $\Lambda$. This suborbifold is commensurable with the modular surface and so by Corollary \ref{Cor:CommInv} contains arithmetic progressions. As $\B{G}_0 \cap \Lambda$ is separable in $\Lambda$ (see \cite[Prop 3.8]{McReynolds} and the references therein), there exists a finite index subgroup $\Lambda_0 < \Lambda$ such that $\B{G}_0 \cap \Lambda < \Lambda_0$ and the induced isometric inclusion of the orbifold associated to $\B{G}_0 \cap \Lambda$ embeds in the locally symmetric orbifold associated to $\Lambda_0$. By Lemma \ref{tot-geod-gives-AP-general}, the orbifold associated to $\Lambda_0$ has arithmetic progressions and by Proposition \ref{CommInv}, the orbifold associated to $\Lambda$ has arithmetic progressions.
\end{proof}

\subsection{Proof of Theorem \ref{Arith3Man}}

For a number field $K/\Q$, we can consider the groups $\PSL(2,\Cal{O}_K)$. If $K$ has $r_1$ real places and $r_2$ complex places, up to conjugation, then we define $X_K = ((\Hy^2)^{r_1} \times (\Hy^3)^{r_2})/\PSL(2,\Cal{O}_K)$. The spaces $X_K$ are non-compact, locally symmetric, arithmetic orbifolds. When $K$ is a real quadratic field, these orbifolds $X_K$ are called \textbf{Hilbert modular surfaces}. When $K$ is an imaginary quadratic field, the groups $\PSL(2,\Cal{O}_K)$ are called \textbf{Bianchi groups} and the associated orbifolds $X_K$ are non-compact, arithmetic, hyperbolic $3$--orbifolds.

Corollary \ref{AllPrimArise} holds for the non-compact, locally symmetric, arithmetic orbifolds $X_K$. We again have a function 
\[ P\colon \PSL(2,\Cal{O}_K) \times \PGL(2,K) \longrightarrow \N \] 
given by 
\[ P(\gamma,\eta) = \min\set{j \in \N~:~\eta \gamma^j \eta^{-1} \in \PSL(2,\Cal{O}_K)}. \] 
The general methods used for $\PSL(2,\Z)$ can then be used in this setting to prove the strong form that every primitive length arises in arbitrarily long arithmetic progressions. Important here is that we still have the exact sequence (\ref{ExactSequence}). To be explicit, taking a prime ideal $\Fr{p}$ in $\Cal{O}_K$, we have the exact sequence
\[ 1 \lra V_\Fr{p} \lra \PSL(2,\Cal{O}_K/\Fr{p}^{j+1}) \lra \PSL(2,\Cal{O}_K/\Fr{p}^j) \lra 1 \]
where $(V_\Fr{p},+)$ is a 3--dimensional $(\Cal{O}_K/\Fr{p})$--vector space; note $(V_\Fr{p},+)$ has exponent $p$ where $p$ is the characteristic of the finite field $\Cal{O}_K/\Fr{p}$. We can also conjugate by elements of the form $\eta_\alpha$ for $\alpha \in K^\times$. Since every non-compact, arithmetic orbifold modeled on $((\Hy^2)^{r_1} \times (\Hy^3)^{r_2})$ is commensurable with $X_K$ for some number field $K$ with $r_1$ real places and $r_2$ complex places, the above in tandem with Corollary \ref{Cor:CommInv} proves all of these orbifolds satisfy the strong form for arithmetic progressions. 

\begin{cor}\label{GenArith3Man}
If $M$ is a non-compact, arithmetic orbifold modeled on $((\Hy^2)^{r_1} \times (\Hy^3)^{r_2})$, then every primitive length occurs in arithmetic progressions.
\end{cor}

When $r_1=0$ and $r_2=1$, we obtain Theorem \ref{Arith3Man}. When $r_1+r_2>1$, the arithmeticity assumption is unnecessary by Margulis' arithmeticity theorem.

\section{Final remarks}\label{Final}

\subsection{Conjectural characterization of arithmeticity}

In this article, we have shown that for negatively curved metrics, despite the fact that almost arithmetic progressions are abundant, genuine arithmetic progressions are rare. We have provided several examples of arithmetic negatively curved (and non-positively curved) manifolds which have arithmetic progressions. It is tempting to conjecture that {\it all} arithmetic manifolds have arithmetic progressions. In fact, we have little doubt that this holds. It is tempting to conjecture that the presence of arithmetic progressions in the primitive length spectrum can be used to characterize arithmetic manifolds. However, one should be a bit careful. Using Corollary \ref{tot-geod-gives-AP}, one can easily produce examples of non-arithmetic, negatively curved manifolds whose length spectrum has arithmetic progressions. Start with a high-dimensional hyperbolic manifold $M$ which contains a non-compact arithmetic hyperbolic surface as a totally geodesic submanifold $N$; every non-compact, arithmetic hyperbolic $n$--manifold has such a surface (see, for instance, Theorem 5.1 in \cite{McReynolds} for a description of the non-compact arithmetic lattices in $\Isom(\Hy^n)$). Pick an arbitrary point $p\in M\setminus N$, and slightly perturb the metric in a small enough neighborhood of $p$. If the perturbation is small enough, the resulting Riemannian manifold $(M, g)$ will still be negatively curved, though no longer hyperbolic. Since the perturbation is performed away from the submanifold $N$, the latter will still be totally geodesic inside $(M,g)$. So Corollary \ref{tot-geod-gives-AP} ensures that the resulting $\mathcal L_p(M,g)$ has arithmetic progressions, even though $(M,g)$ is not arithmetic (in fact, not even locally symmetric). One simple result of this discussion is the following:

\begin{cor}
The set of metrics whose primitive length spectrum have arithmetic progressions is not discrete.
\end{cor}

Note that the non-arithmetic examples of Gromov--Piatetski-Shapiro \cite{GPS} are built by gluing together two 
arithmetic manifolds along a common totally geodesic hypersurface. Being arithmetic, this hypersurface contains arithmetic progressions, and from our Lemma \ref{tot-geod-gives-AP-general}, the hybrid non-arithmetic manifold would then also have arithmetic progressions. Reid \cite[Thm 3]{Reid2} constructed infinitely many commensurability classes of non-arithmetic hyperbolic 3--manifolds, which are hyperbolic knot complements in $S^3$ with a unique commensurability class of immersed totally geodesic surfaces. All of these surfaces cover the modular surface and hence by Corollary \ref{tot-geod-gives-AP}, these non-arithmetic hyperbolic 3--manifolds have arithmetic progressions.

However, recall that our constructions actually show that the arithmetic manifolds we consider satisfy a much stronger condition than just having arithmetic progressions. Namely, {\it every primitive geodesic length occurs in arithmetic progressions}. The hybrid manifolds of Gromov--Piatetski-Shapiro are unlikely to satisfy this much stronger condition, as a generic primitive geodesic is unlikely to reside on an arithmetic submanifold. Indeed recent work of Fisher--Lafont--Miller--Stover \cite{FLMS} shows such $n$--manifolds contain only finitely many closed totally geodesic submanifolds of dimension $2\leq k \leq n-1$ that are maximal (in terms of containment).
In particular, it is unclear where one might find infinitely many primitive geodesics that have the same length (up to rational multiples) as our given primitive geodesic. 


\underline{\textbf{Conjecture A.}} \textsl{Let $(M,g)$ be a closed or finite volume, complete Riemannian manifold. If 
$\mathcal L_p(M,g)$ has every primitive length occurring in arithmetic progressions (in the sense of Section 
\ref{gen-and-imp}), then $(M,g)$ is arithmetic.}

A much weaker version of Conjecture A, where we restrict the topological type of the manifold $M$, would already
be of considerable interest:

\underline{\textbf{Conjecture B.}} \textsl{Let $M$ be a closed manifold that admits a locally symmetric metric, and
assume that the universal cover of $M$ has no compact factors and $M$ is irreducible. Given a metric $(M,g)$ on $M$,
assume that $\mathcal L_p(M,g)$ has every primitive length occurring in arithmetic progressions (in the sense of 
Section \ref{gen-and-imp}). Then $g$ is a locally symmetric metric, and is arithmetic.}

At present, it is still an open problem as to whether higher rank, locally symmetric manifolds $(M,g_{sym})$ 
are determined in the space of Riemannian metrics by their primitive length spectrum. The local version of this type of rigidity is often referred to as \textbf{spectral isolation}. The spectral isolation of symmetric or locally symmetric metrics seems to be a folklore conjecture that has been around for some time; see \cite{GSS} for some recent work and history on this problem. Conjecture B implies the stronger global spectral rigidity conjecture immediately for locally symmetric metrics; one might say the locally symmetric metric is \textbf{spectrally isolated globally} in that case. Our last conjecture is weaker than Conjecture A and B.

\underline{\textbf{Conjecture C.}} \textsl{Let $M$ be a closed manifold that admits a negatively curved metric and let 
$\Cal{M}(M)$ denote the space of negatively curved metrics with the Lipschitz topology. Consider the metrics 
with the property that $\mathcal L_p(M,g)$ has every primitive length occurring in arithmetic progressions
(in the sense of Section \ref{gen-and-imp}). Then the set of such metrics forms a discrete (or even better, finite) 
subset of $\Cal{M}(M)$.}

We do not know whether Conjecture C holds when $M$ is a closed surface of genus at least two. Higher genus closed surfaces are a test case for this conjecture.

\subsection{Other proposed characterizations of arithmeticity}

Sarnak \cite{Sarnak1} proposed a characterization for arithmetic surfaces that is also of a geometric nature. For a Fuchsian group $\Gamma < \PSL(2,\R)$, set $\Tr(\Gamma) = \set{\abs{\Tr(\gamma)}~:~\gamma \in \Gamma}$. A Fuchsian group satisfies the \textbf{bounded clustering property} if there exists a constant $C_\Gamma$ such that, for all integers $n$, we have $\abs{\Tr(\Gamma) \cap [n,n+1]} < C_\Gamma$. It was verified by Luo--Sarnak \cite{LuoS} that arithmetic surfaces satisfy the bounded clustering property. Schmutz \cite{Schmutz} proposed a characterization of arithmeticity based on the function $F(x) = \abs{\Tr(\Gamma) \cap [0,x]}$. Specifically, $\Gamma$ is arithmetic if and only if $F(x)$ grows at most linearly in $x$. Geninska--Leuzinger \cite{GL} verified Sarnak's conjecture in the case where $\Gamma$ contains a non-trivial parabolic isometry. 
In \cite{GL}, they also point out a gap in \cite{Schmutz} that verified the linear growth characterization for lattices with a non-trivial parabolic isometry. At present, this verification seems to still be open. These characterizations of arithmeticity are based on the fact that arithmetic manifolds have unusually high multiplicities in the primitive geodesic length spectrum, a phenomenon first observed by Selberg. One explanation for the high multiplicities can be seen from our proof that arithmetic, non-compact surfaces have arithmetic progressions. Specifically, from one primitive length $\ell$, via the commensurator, we can produce infinitely many primitive lengths of the form $\pr{\frac{m}{d}}\ell$, where $m$ ranges over an infinite set of integers and $d$ ranges over a finite set of integers. When $\ell$ is the associated length of an absolutely primitive element, we obtain lengths of the form $m\ell$ as $m$ ranges over an infinite set of integers. Given the freedom on the production of these lengths, it is impossible to imagine that huge multiplicities will not arise. Other characterizations of arithmeticity given by Cooper--Long--Reid \cite{CLR} (see also Reid \cite{Reid}) and Farb--Weinberger \cite{FW} exploit the abundant presence of symmetries, and thus are still in the realm of Margulis' characterization via commensurators. Reid \cite{Reid1}, Chinburg--Hamilton--Long--Reid \cite{CHLR}, and Prasad--Rapinchuk \cite{PrRap} also recover arithmeticity using spectral invariants, and so we feel our proposed characterization sits somewhere between the commensurator and spectral sides.




\begin{thebibliography}{9}

\bibitem{Abraham}
R.~Abraham, \emph{Bumpy metrics}, Proc. Sympos. Pure Math. \textbf{XIV} (1970), 1--3.

\bibitem{Bass}
H.~Bass, \emph{Algebraic $K$--theory}, Springer-Verlag, 1968.

\bibitem{BHC}
A.~Borel, Harish-Chandra, \emph{Arithmetic subgroups of algebraic groups}, Ann. of Math. \textbf{75} (1962), 485--535.

\bibitem{CLM}
V.~Chernousov, L.~Lifschitz, D.~Witte--Morris, \emph{Almost-minimal nonuniform lattices of higher rank}, Michigan Math. J. \textbf{56} (2008), 453--478.

\bibitem{CHLR}
T.~Chinburg, E.~Hamilton, D.~D.~Long, A.~W.~Reid, \emph{Geodesics and commensurability classes of arithmetic hyperbolic $3$--manifolds}, Duke Math. J. \textbf{145} (2008), 25--44.


\bibitem{CLR} 
D.~Cooper, D.~D.~Long, and A.~W.~Reid, \emph{On the virtual Betti numbers of arithmetic hyperbolic 3-manifolds}, Geom. Topol. \textbf{11} (2007), 2265--2276.

\bibitem{CLT}
D.~Constantine, J.-F.~Lafont, D.J.~Thompson, \emph{The weak specification property for geodesic flows on CAT(-1) spaces}, \href{https://arxiv.org/abs/1606.06253}{ArXiV}.

\bibitem{Dietmar}
A.~Deitmar, \emph{A prime geodesic theorem for higher rank spaces}, Geom. Funct. Anal. \textbf{14} (2004), 1238--1266.

\bibitem{DSMS}
J.~D.~Dixon, M.~P.~F.~du Sautoy, A.~Mann, D.~Segal, \emph{Analytic pro--$p$
Groups}, Cambridge Studies in Advanced Maths. \textbf{61}, 1999.

\bibitem{E} 
P.~Eberlein, \emph{When is a geodesic flow of Anosov type? I, II}, J. Diff. Geom. \textbf{8} (1973), 437--463, 565--577.

\bibitem{FW} 
B.~Farb, S.~Weinberger, \emph{Isometries, rigidity and universal covers}, Ann. of Math. \textbf{168} (2008), 915--940.

\bibitem{FLMS}
D.~Fisher, J.-F.~Lafont, N.~Miller, M.~Stover, \emph{Finiteness of maximal geodesic submanifolds in hyperbolic hybrids},
\href{https://arxiv.org/abs/1802.04619}{ArXiv}.


\bibitem{GL}
S.~Geninska, E.~Leuzinger, \emph{A geometric characterization of arithmetic Fuchsian groups}, Duke Math. J. \textbf{142} (2008), 111--125. 

\bibitem{Godement}
R.~Godement, \emph{Domaines fondamentaux des groupes arithm\'{e}tiques}, S\'{e}minaire Bourbaki, 1962/63. Fasc. 3, No. 257 25 pp. Secr\'{e}tariat math\'{e}matique, Paris

\bibitem{GSS}
C.~S.~Gordon, D.~Schueth, C.~J.~Sutton, \emph{Spectral isolation of bi-invariant metrics on compact Lie groups}, Ann. Inst. Fourier \textbf{60} (2010), 1617--1628.

\bibitem{GR}
R.~L.~Graham, B.~L.~Rothschild, \emph{A short proof of van der Waerden's theorem on arithmetic progressions}, Proc. Amer. Math. Soc. \textbf{42} (1974), 385--386.

\bibitem{GPS}
M.~Gromov, I.~Piatetski-Shapiro, \emph{Nonarithmetic groups in Lobachevsky spaces}, Inst. Hautes \'{E}tudes Sci. Publ. Math. \textbf{66} (1988), 93--103.

\bibitem{hasselblatt-katok}
B.~Hasselblatt, A.~Katok, \emph{Introduction to the modern theory of dynamical systems}, Cambridge University
Press, 1995.

\bibitem{Huber}
H.~Huber, \emph{Zur analytischen theorie hyperbolischer Raumformen und Bewegungsgruppen. II}, Math. Ann. \textbf{143} (1961), 463--464.

\bibitem{Jacobson}
N.~Jacobson, \emph{Completely reducible Lie algebras of linear transformations}, Proc. Amer. Math. Soc. \textbf{2} (1951), 105--113.

\bibitem{LS}
A.~Lubtozky, D.~Segal, \emph{Subgroup growth}, Birkh\"{a}user, 2003.

\bibitem{LuoS}
W.~Luo, P.~Sarnak, \emph{Number variance for arithmetic hyperbolic surfaces}, Comm. Math. Phys. \textbf{161} (1994), 419--432. 

\bibitem{Marcus}
D.~A.~Marcus, \emph{Number fields}, Springer--Verlag, 1977.

\bibitem{MR}
C.~Maclachlan, A.~W.~Reid, \emph{The arithmetic of hyperbolic 3--manifolds}, Springer--Verlag, 2003.

\bibitem{Margulis}
G.~A.~Margulis, \emph{Certain applications of ergodic theory to the investigation of manifolds of negative curvature}, Funkcional. Anal. i Prilo\v{z}en. \textbf{3} (1969), 89--90. 

\bibitem{McReynolds}
D.~B.~McReynolds, \emph{Peripheral separability and cusps of arithmetic hyperbolic orbifolds}, Algebr. and Geom. Topol. \textbf{4} (2004), 721--755.

\bibitem{Miller}
N. Miller, \emph{Arithmetic Progressions in the Primitive Length Spectrum},  \href{https://arxiv.org/abs/1602.01869}{ArXiV}.

\bibitem{Perlis}
R.~Perlis, \emph{On the equation {$\zeta \sb{K}(s)=\zeta \sb{K'}(s)$}}, J. Number Theory \textbf{9} (1977), 342--360.

\bibitem{PR}
V.~Platonov, A.~Rapinchuk, \emph{Algebraic groups and number fields}, Academic Press, 1994.

\bibitem{PS}
M.~Pollicott, R.~Sharp, \emph{Exponential error terms for growth functions on negatively curved surfaces}, Amer. J. Math. \textbf{120} (1998), 1019--1042.


\bibitem{PrRap}
G.~Prasad, A.~S.~Rapinchuk, \emph{Weakly commensurable arithmetic groups and isospectral locally symmetric spaces}, Publ. Math. Inst. Hautes \'{E}tudes Sci. \textbf{109} (2009), 113--184. 

\bibitem{Rag}
M.~S.~Raghunathan, \emph{Discrete subgroups of Lie groups}, Springer--Verlag, 1972.

\bibitem{Reid2}
A.~W.~Reid, \emph{Totally geodesic surfaces in hyperbolic 3--manifolds}, Proc. Edinb. Math. Soc. \textbf{34} (1991), 77--88.

\bibitem{Reid1}
A.~W.~Reid, \emph{Commensurability and isospectrality of arithmetic hyperbolic $2$ and $3$--manifolds}, Duke Math J. \textbf{65} (1992), 215--228.

\bibitem{Reid} 
A.~W.~Reid, \emph{The geometry and topology of arithmetic hyperbolic 3--manifolds}, Topology, complex analysis, and arithmetic of hyperbolic spaces, RIMS 1571 (2007), 31--58.

\bibitem{Sarnak}
P.~C.~Sarnak, \emph{Prime geodesic theorems}, Thesis (Ph.D., 1980) Stanford University, 111 pp.

\bibitem{Sarnak1}
P.~C.~Sarnak, \emph{Arithmetic quantum chaos}, The Schur lectures (1992) (Tel Aviv), 183--236, 
Israel Math. Conf. Proc., 8, Bar-Ilan Univ., Ramat Gan, 1995. 

\bibitem{Schmutz}
P.~Schmutz, \emph{Arithmetic groups and the length spectrum of Riemann surfaces}, Duke Math. J. \textbf{84} (1996), 199--215.

\bibitem{SY}
K.~Soundararajan, M.~P.~Young, \emph{The prime geodesic theorem}, J. Reine Angew. Math. \textbf{676} (2013), 105--120.

\bibitem{Stopple}
J.~Stopple, \emph{A reciprocity law for prime geodesics}, J. Number Theory \textbf{29} (1988), 224--230.

\bibitem{Sunada}
T.~Sunada, \emph{Riemannian coverings and isospectral manifolds}, Ann. of Math. \textbf{121} (1985), 169--186.

\bibitem{Sunada2}
T.~Sunada, \emph{Number theoretic analogues in spectral geometry}, Differential geometry and differential equations (Shanghai, 1985), 96--108, Lecture Notes in Math., 1255, Springer, Berlin, 1987.

\bibitem{Sz}
E.~Szemer\'edi, \emph{On sets of integers containing no $k$ elements in arithmetic progression}, Acta Arith. \textbf{27} (1975), 199--245.

\bibitem{van}
B.~L.~van der Waerden, \emph{Beweis einer Baudetschen Vermutung}, Nieuw. Arch. Wisk. \textbf{15} (1927), 212--216.

\bibitem{Witte}
D.~Witte-Morris, \emph{An introduction to arithmetic lattices}, Deductive Press, 2015.

\end{thebibliography}
\end{document}